\documentclass[11pt]{amsart}

\usepackage[utf8]{inputenc}
\usepackage[T1]{fontenc}
\usepackage[french,british]{babel}
\usepackage{lmodern}
\usepackage[a4paper]{geometry}
\usepackage{amsmath,amssymb,amsthm}
\usepackage[shortlabels]{enumitem}
\usepackage[all]{xy}
\usepackage{multirow}
\usepackage{caption}

\theoremstyle{plain}
\newtheorem{theo}{Theorem}[section]
\newtheorem{lem}[theo]{Lemma}
\newtheorem{cor}[theo]{Corollary}
\newtheorem{prop}[theo]{Proposition}
\theoremstyle{definition}
\newtheorem{defi}[theo]{Definition}

\newtheorem{rem}[theo]{Remark}

\theoremstyle{plain}

\theoremstyle{definition}

\theoremstyle{plain}
\newtheorem{theorem}{Theorem}

\title[Rationality of the Weil representation and the theta correspondence]{On the rationality of the Weil Representation and the local theta correspondence}
\author{Justin Trias}
\date{}

\begin{document}

\maketitle

\begin{abstract} We prove that the Weil representation over a non-archimedean local field can be realised with coefficients in a number field. We give an explicit descent argument to describe precisely which number field the Weil representation descends to. Our methods also apply over more general coefficient fields, such as $\ell$-modular coefficient fields, as well as coefficient rings such as rings of integers \textit{i.e.} in families. We also prove that the theta correspondence over a perfect field is valid if and only if it is valid over the algebraic closure of this perfect field. These two results together show that the classical local theta correspondence is rational. \end{abstract}

\tableofcontents

\section*{Introduction}

In Remark 2 of his note \textit{A Brief Survey on the Theta Correspondence}, Dipendra Prasad expresses the following expectation
\begin{quote}
\textit{It will be interesting to construct a model of the Weil representation which is defined over a number field. Since all the known models require the additive character $\psi$ in an essential way, it does not seem obvious if it can be done at all.}
\end{quote}
And he continues
\begin{quote}
\textit{We note that the Weil representation of $\textup{SL}(2)$ can be defined over a number field because it is sum of its even and odd pieces, both of which are defined over number fields: the even piece because it occurs in an explicit principal series, and the odd piece because it is induced from compact open subgroup.}
\end{quote}
One of the main goals of this paper is to address this question and define the Weil representation over a number field. This is achieved in the second part of the manuscript. To do so, we perform an explicit Galois descent on the Weil representation to obtain a model defined over a number field. We work in the finite case and in the non-archimedean local case \textit{i.e} over a field $F$ of characteristic not $2$ that is either local non-archimedean of residual cardinality $q = p^f$ or finite with cardinality $q=p^f$. Because our methods are explicit, we describe the minimal number fields over which the Weil representation can be realised in the following table. We let $p^*$ be $-p$ if $p \equiv 3 [4]$ and $p$ if $p \equiv 1 [4]$. We denote by $\omega^+$ the even part and by $\omega^-$ the odd part of the Weil representation. Here is a summary of the results we obtain in Section \ref{sec:descent-when-p-not-2}. 

\begin{theorem}[Character fields and realisation fields when $p \neq 2$] \ \\
\begin{center} \begin{tabular}{|c|c|c|}
\hline
\textup{Weil rep.} & $\omega^+$ & $\omega^-$ \\
\hline
\multirow{2}{*}{\textup{char. field}} & \multicolumn{2}{c|}{$\mathbb{Q}$ \textup{if} $q \in p^{2\mathbb{N} \phantom{e^2}}$} \\
 & \multicolumn{2}{c|}{$\mathbb{Q}[\sqrt{p^*}]$ \textup{if} $q \in p^{2\mathbb{N}+1}$} \\
\hline
\multirow{2}{*}{\textup{real. field}} & $\mathbb{Q}$ \textup{if} $q \in p^{2\mathbb{N} \phantom{e^2}}$ & $\mathbb{Q}[\sqrt{-p}]$ \textup{if} $q \in p^{2\mathbb{N}}$ \\
 & \hspace{0.1cm} $\mathbb{Q}[\sqrt{p^*}]$ \textup{if} $q \in p^{2\mathbb{N}+1}$ \hspace{0.1cm} & $\mathbb{Q}[\sqrt{p^*}][\sqrt{-p}]$ \textup{if} $q \in p^{2 \mathbb{N}+1}$ \\
\hline
\end{tabular} \end{center}
\end{theorem}

\vspace{0.2cm}

\noindent We note that these fields do not depend on the size of the symplectic group on which the Weil representation is built and only depend on $q$. We also obtain explicit results when $p=2$ but there are more subcases which not only depend on $q$ but also on the field $F$ itself, so we do not explain them here and rather refer to Theorems \ref{thm:descent-even-part-p-is-2} and \ref{thm:descent-odd-part-p-is-2}. The character and realisation fields of the Weil representation are obtained as the composite of the character and realisation fields of the even and odd parts.

When our paper was finished, Dipendra Prasad kindly informed us of the existence of two papers unknown to the author on the rationality of the Weil representation, one in the finite case \cite[Sec 13]{gross} and another one in the non-archimedean case \cite{cliff_mc_neilly}. They both assume $p$ is odd. It should be noted that they do not perform an explicit descent, as we do in this paper. Therefore they can't obtain an explicit model of the Weil representation over a number field, but this can be extracted from our Galois descent data by Theorem \ref{thm:weil-representation-descent-p-not-2} and the explicit obstruction norm problem in Lemma \ref{lem:descent-impossible-along-mathcal-L}. Moreover \cite{cliff_mc_neilly} does not describe the Schur index in all cases. Our explicit descent relies on an interpretation of the Weil representation in \cite{trias_modular_weil} which is different from \cite{cliff_mc_neilly} and is simpler to manipulate than the classical Schr\"odinger model they use, which is realised over $\mathbb{Q}(\psi,\sqrt{-1})$, where $\mathbb{Q}(\psi)$ the character field of a non-trivial smooth character $\psi : F \to \mathbb{C}^\times$, whereas the model we use is directly realised over $\mathbb{Q}(\psi)$.

We now expose the main results of the first part of the manuscript, where we study some generalities about rationality in the representation theory of locally profinite groups and prove a result about the rationality of the largest isotypic quotient \textit{i.e.} a compatibility between the largest istoypic quotients over a perfect field $R$ and over its algebraic closure. This applies in particular to theta lifts, which are obtained as largest isotypic quotients. We then define the local theta correspondence over $R$ in Section \ref{sec:isotypic_lifts_and_rationality} as a set of statements about finiteness, irreducibilty and uniqueness of the theta lifts. We prove in Theorem \ref{thm:theta_corresp_over_R_and_bar_R} that the local theta correspondence is valid over $R$ if and only if it is valid over its algebraic closure. The other main result we obtain is a compatibility of the theta lifts with Galois action in a sense we now explain. Let $\psi : F \to R^\times$ be a non-trivial smooth character and let $\omega_\psi$ be the Weil representation. For $a \in F^\times$, let $\psi^a : t \in F \mapsto \psi(at) \in R^\times$, which is a non-trivial character. We denote by $[\psi]$ the orbit of the character $\psi$ under the action of $F^{\times 2}$. Then $\omega_\psi \simeq \omega_{\psi'}$ if and only if $\psi' \in [\psi]$. We denote by $\Theta_{[\psi]}$ the theta lift to insist on the dependence in the character $\psi$. Let $H_1$ and $H_2$ be a reductive dual pair in a symplectic group, or its lifts to the metaplectic group to be more rigorous. Here is the compatibility we obtain at the end of the first part:

\begin{theorem}[Theorem \ref{thm:theta_lifts_and_galois_action}] Let $R = \overline{\mathbb{Q}}$ and let $\sigma \in \textup{Gal}(\overline{\mathbb{Q}}/\mathbb{Q})$. Then $\Theta_{[\psi]}$ is equivariant for the action of the Galois group in the sense that, for all $\pi_1 \in \textup{Irr}_{\overline{\mathbb{Q}}}(H_1)$, we have
$${}^\sigma \Theta_{[\psi]}(\pi_1) \simeq \Theta_{[\psi^\sigma]}({}^\sigma \pi_1).$$
\end{theorem}
 
There are several perspectives we would like to explore after this work, that appeared as some sources of motivation at a late stage of writing, such as questions related to the global theta correspondence and its rationality -- the rationality of certain automorphic periods seems of great interest and there is a series of papers  \cite{prasanna1,prasanna2,prasanna3} by Prasanna in this direction in the context of the Shimura correspondence and  the work of Waldspurger. There are also two additional questions that we do not address in this work and intend to study later. On the one hand, we simply bound the field of realisation in the local theta correspondence as it could shrink even more by pulling back the Weil representation to a specific dual pair, though our result is optimal for the pair $(\textup{Sp}(W),\{\pm 1 \})$. On the other hand, we did not consider the local archimedean version of the Weil representation. Though its usual archimedean model may not be well suited for a descent argument, its Fock model is more likely to be. Finally, one should be very cautious towards splittings in the theta correspondence as they may require to introduce some $4$-th or $8$-th roots of unity that could take us out of the realisation field.

\subsection*{Content of the paper} In the first section, we recall the definition of the character field of a representation as well as the notion of a field of realisation. In the second and third sections, we develop some general background about representations with coefficients in a perfect field that is not necessarily algebraically closed. We explain how extension and restriction of scalars behave and compare the largest isotypic quotients over a base field and its algebraic closure. In the fourth section, we apply the results of the previous sections in the context of theta lifts. We also introduce the Schur index. In the fifth section, we expose the Galois descent theorems we are going to use and we rephrase them in terms of Morita equivalences -- representation theorists may find this perspective appealing. The sixth section points out that the representation which is at the heart of the construction of the Weil representation, namely the Heisenberg representation, does not descend to a number field. This could look like a negative answer to descending the Weil representation over a number field, but the answer is a bit more subtle because, as we show in the last two sections, the Weil representation does descend to a number field even though the Heisenberg representation does not. This is, in some sense, the obstacle Dipendra Prasad already remarked by referring to the additive character $\psi$. The seventh and eighth sections are dedicated to performing our explicit Galois descent on the Weil representation. On top of finding the character field and the fields of realisation for the -- even and odd parts of the -- Weil representation, we are also able to determine its Schur index. Our results also apply in the modular setting -- \textit{i.e.} for coefficient fields of positive characteristic -- and in families -- \textit{i.e.} for coefficient rings such as rings of integers.

\addtocontents{toc}{\protect\setcounter{tocdepth}{0}}

\section*{Acknowledgements}

This paper would likely have remained at the stage of an interesting future project without a question raised in a discussion with Daniel Designi. His interest in the topic and his encouragement truly energized the author and helped bring this project to fruition as the present article. No expression of thanks can fully capture my appreciation for him. I am also indebted to Alberto M\'inguez and Shaun Stevens for their constant support and for their helpful comments on the final draft of this paper. Moreover, I greatly benefited from fruitful discussions with Petar Baki\'c, Raphaël Beuzart-Plessis, Marcela Hanzer, Gil Moss and Jack Sempliner. Finally, the short surveys by Dipendra Prasad on the theta correspondence have long been a great source of inspiration, especially during my PhD, and I remain deeply grateful each time I return to them. I would also like to express my gratitude to him for a stimulating conversation on the final draft of this paper and for pointing out two already existing papers in the topic.

The author was partially supported by the EPSRC Grant EP/V061739/1. This work was partly funded by the European Union ERC Consolidator Grant, RELANTRA, project number 101044930. Views and opinions expressed are however those of the author only and do not necessarily reflect those of the European Union or the European Research Council. Neither the European Union nor the granting authority can be held responsible for them. 

\addtocontents{toc}{\protect\setcounter{tocdepth}{1}} 

\section*{Notations}

A locally profinite group $G$ is a locally compact totally disconnected topological group. Let $K$ be a compact open subgroup of $G$. The pro-order of $K$ is the least common multiple of the cardinality of the finite quotients of $K$ \cite[I.1.5]{vig}. The pro-order $|G|$ of $G$ is the least common multiple of the $|K|$'s where $K$ runs over all compact open subgroups of $G$. When $G$ is a reductive group over $F$, \textit{i.e.} the $F$-points of a reductive algebraic group defined over $F$, we usually have $|G| =  n_f p^k$ where $n_f \in \mathbb{N}$ is prime-to-$p$ and $k \in \mathbb{N} \cup \{\infty\}$.

Let $R$ be a commutative ring. Let $C_c^\infty(G,R)$ be the space of locally constant compactly supported functions on $G$ valued in $R$. If $G$ contains an open subgroup of invertible pro-order in $R$, there exists a Haar measure $\mu$ of $G$ with values in $R$ by \cite[I.2.4]{vig}. If a compact open subgroup $K$ has invertible pro-order in $R$, there exists a unique measure $\mu_K$ such that $K$ has volume $1$. We call it the normalised measure on $K$. All such normalised measures are unique up to a scalar in $R^\times$ and the normalised measures generate all Haar measures on $G$. After fixing a normalised measure of $G$, we can endow $C_c^\infty(G,R)$ with a structure of $R$-algebra and we denote this algebra by $\mathcal{H}_R(G)$ and call it the Hecke algebra.

An $R[G]$-module $V$ is smooth if $\textup{Stab}_G(v) = \{ g \in G \ | \ g \cdot V\}$ is open in $G$ for all $v \in V$. We also use the word representation for a smooth $R[G]$-module. We denote by $\textup{Rep}_R(G)$ the category of smooth $R[G]$-modules. We say a representation is admissible if the subspace of $K$-fixed vectors $V^K$ is a finitely generated $R$-module. Let $H$ be a closed subgroup of $G$, we define a functor $\textup{Ind}_H^G : \textup{Rep}_R(H) \to \textup{Rep}_R(G)$ where for $\sigma \in \textup{Rep}_R(H)$, we associate the space $\textup{Ind}_H^G(\sigma)$ of functions $f : G \to \sigma$ such that $f(hg) = \sigma(h) f(g)$ and $f$ is smooth, endowed with the  smooth $G$-action $g \cdot f(g') = f(g' g)$. We also define the subfunctor $\textup{ind}_H^G$ of $\textup{Ind}_H^G$ by moreover requiring that $f$ has compact support modulo $H$.

For $n \in \mathbb{N}$, we denote by $\zeta_n \in \mathbb{C}$ the usual primitive $n$-root of unity \textit{i.e.} $\zeta_n = e^{\frac{2 i \pi}{n}}$. If there exists a non-trivial smooth (additive) character $\psi : F \to R^\times$, then necessarily the characteristic $\ell$ of $R$ is different from $p$. Moreover $R$ must contain enough $p$-roots or $p$-power roots of unity. Let $\mathbb{Z}[\zeta_{p^\infty}] = \cup_k \mathbb{Z}[\zeta_{p^k}]$ and let
$$\mathcal{A} = \left\{ \begin{array}{cc}
\mathbb{Z}[\frac{1}{p},\zeta_{p^\infty}] & \textup{ if } \textup{char} (F) = 0; \\
 & \\
\mathbb{Z}[\frac{1}{p},\zeta_p] & \textup{ if } \textup{char}(F) > 0.
\end{array} \right.$$
Then there exists a non-trivial character $\psi : F \to R^\times$ if and only if $R$ can be endowed with a structure of $\mathcal{A}$-algebra. We always assume $R$ satisfies this condition.

Let $F$ be a field of characteristic different from $2$, that is either a finite field of cardinality $q$ or a non-archimedean local field of residue characteristic $q$. We write $q=p^f$. When $F$ is local non-archimedean, we let $\mathcal{O}_F$ be its ring of integers and $k_F$ its residue field and we fix a uniformiser $\varpi_F$ in $\mathcal{O}_F$. Let $( \ , \ )_F$ be the quadratic Hilbert symbol, which is trivial if $F$ is finite. If $F$ is local non-archimedean and $V$ is a finite dimensional $F$-vector space, a lattice in $V$ is a free $\mathcal{O}_F$-module of rank the dimension of $V$.

Let $(W, \langle \ , \ \rangle)$ be a symplectic vector space of dimension $n=2m$ over $F$. A subspace $X \subseteq W$ is totally isotropic if $\langle \ , \ \rangle|_{X \times X}$ is identically zero. A totally isotropic subspace is maximal if and only if it has dimension $m$. Such a maximal space is called a Lagrangian in $W$. A complete polarisation $W=X \oplus Y$ is made of two transverse Lagrangians $X$ and $Y$ in $W$. The symplectic group $\textup{Sp}(W)$ is the group of isometries of $W$.

\part{Generalities about rationality}

\section{Rationality} \label{sec:rationality}

\subsection{} Let $R$ be a commutative ring. Let $\mathcal{H}$ be an $R$-algebra. Let $V$ be an $\mathcal{H}$-module that is finite free as an $R$-module. Let $\mathcal{B} = (e_i)_{i \in I}$ be an $R$-basis of $V$ and let $(e_i^*)_{i \in I}$ be its dual basis in $\textup{Hom}_R(V,R)$. This defines endomorphisms $e_{m,n} = e_n^* \otimes_R e_m \in \textup{End}_R(V)$ for $m, n \in \mathcal{B}$ via 
$$e_{m,n}(\sum_{i \in I} r_i e_i) = r_n e_m$$
where $(r_i)_{i \in I}$ are elements in $R$.  

For $h \in \mathcal{H}$, we denote by $h_V \in \textup{End}_R(V)$ the action of $h$ on $V$. There exists a unique family $(\alpha_{m,n})_{m,n \in I}$ of elements of $R$ such that $h_V = \sum_{m,n \in I} \alpha_{m,n} e_{m,n}$. We define $\textup{tr}_V : \mathcal{H} \to R$ by
$$\textup{tr}_V (h) = \sum_{m \in I} \alpha_{m,m}.$$
It is well-known that $\textup{tr}_V$ is actually independent of the choice of the basis $\mathcal{B}$.

\begin{defi} We call $\textup{tr}_V$ the trace-character of $V$. \end{defi}

\subsection{} From now on, let $R_0$ be $\mathbb{Q}$ or $\mathbb{F}_\ell$ and assume $R$ is an algebraic extension of $R_0$. We define $R_c(V)$ as the subfield of $R$ generated by $R_0$ and the values of $\textup{tr}_V$.

\begin{defi} We call $R_c(V)$ the character field of $V$ and we call any subfield of $R$ containing $R_c(V)$ a field of character of $V$. \end{defi} 

\subsection{} Let $\textup{Aut}(R)$ be the automorphism of fields of $R$. Such automorphisms are necessarily $R_0$-linear, so this is also $\textup{Aut}_{R_0}(R)$. For $Z$ a subset of $\textup{Aut}(R)$, we denote by $R^Z$ the subfield of $R$ formed by the elements fixed by $Z$.

For $\sigma \in \textup{Aut}(R)$, let $f_\sigma :  V \to V$ be the $\sigma$-equivariant isomorphism defined by
$$f_\sigma(\sum r_i e_i) = \sum \sigma(r_i) e_i.$$
We endow the $R$-vector space $V^{f_\sigma} = V$ with the action $\mathcal{H}$-action defined by
$$h_{V^{f_\sigma}} = f_\sigma h_V f_\sigma^{-1} \in \textup{End}_R(V) \textup{ for } h \in \mathcal{H}.$$
If $\mathcal{B}' = (e_i')$ is another $R$-basis of $V$, we can define $f_\sigma'$ and $V^{f_\sigma'}$ in an analogous way. The two $\mathcal{H}$-modules $V^{f_\sigma}$ and $V^{f_\sigma'}$ thus obtained are isomorphic. We denote by $V^\sigma$ this well-defined isomorphism class. 

Let $H(V) = \{ \sigma \in \textup{Aut}(R) \ | \ V^\sigma \simeq V\}$ and set $R_r(V) = R^{H(V)}$.

\begin{defi} We call $R_r(V)$ the rationality field of $V$ and we call any subfield of $R$ containing $R_r(V)$ a field of rationality of $V$. \end{defi}

\subsection{} Let $\bar{R}$ be an algebraic closure of $R$. We say $V$ is absolutely simple if $V \otimes_R \bar{R}$ is a simple $(\mathcal{H} \otimes_R \bar{R})$-module. By the linear independence of characters \cite[A VIII.376]{bou}, we deduce that

\begin{prop} \label{prop:rationality-field-equal-character-field} If $V$ is absolutely simple, then $R_c(V) = R_r(V)$. \end{prop}

In this situation, we simply write $R(V)$ for the character/rationality field of $V$.

\begin{rem} Note that, in positive characteristic $\ell$, we have $\textup{tr}_{\ell V} = \ell \textup{tr}_V$ is always identically zero, so $R_c(\ell V)= R_0$ in this case, though $R_r(\ell V)$ may be a strict extension of $R_0$. Therefore the proposition does not extend to all $V$ semisimple. \end{rem}

\subsection{} Let $\mathcal{H}_0$ be an $R_0$-algebra and assume $\mathcal{H} = \mathcal{H}_0 \otimes_{R_0} R$. The algebraic $\bar{R}$ of $R$ is also an algebraic closure of $R_0$. If $R' \subseteq \bar{R}$ is a subfield, we set $\mathcal{H}_{R'} = \mathcal{H}_0 \otimes_{R_0} R'$.

\begin{defi} We say a subfield $R' \subseteq \bar{R}$ is a field of realisation of $V$ if there exists an $\mathcal{H}_{R'}$-module $V'$ such that $V' \otimes_{R'} \bar{R} \simeq V \otimes_R \bar{R}$ as $\mathcal{H}_{\bar{R}}$-modules. \end{defi}

When $V'$ is realisation of $V$ over $R'$, it is clear that $R_c(V') = R_c(V)$. Therefore a field of realisation of $V$ is always a field of character. In the context of Proposition \ref{prop:rationality-field-equal-character-field}, a field of realisation is also a field of rationality if $V$ is absolutely simple.

\subsection{} We assume that $V$ is absolutely simple. Let $m(V) \in \mathbb{N} \cup \{+ \infty \}$ be the infimum of the degrees $[R':R(V)]$ where $R'$ runs over the fields of realisation of $V$.

\begin{defi} We call $m(V)$ the Schur index of $V$. \end{defi}

\subsection{} \label{sec:trace_of_admissible_reps} When $G$ is a locally profinite group admitting a compact open subgroup of invertible pro-order in $R$, there exists a Haar measure on $G$ with values in $R$ by \cite[I.2]{vig}. This leads to define the Hecke algebra $\mathcal{H}_R(G)$ with coefficients in $R$ by \cite[I.3]{vig}. A smooth representation $V$ of $G$ with coefficients in $R$ is an $R[G]$-module such that $\textup{Stab}_G(v) = \{ g \in G \ | \ g \cdot v = v\}$ is open for all $v \in V$. We denote by $\textup{Rep}_R(G)$ the category of smooth representations and by $\textup{Irr}_R(G)$ the isomorphism classes of irreducible representations.

We say that $V \in \textup{Rep}_R(G)$ is admissible if for all compact open subgroups $K$ of $G$ the vector space $V^K$ of $K$-fixed vectors has finite $R$-dimension. The trace-character $\textup{tr}_V : \mathcal{H}_R(G) \to R$ of admissible representations is defined in \cite[I.6]{vig}. When $K$ is a compact open subgroup of $G$, its restriction to the relative Hecke algebras $\mathcal{H}_R(G,K)$ gives a trace-character as considered above. If $V \in \textup{Rep}_R(G)$ is admissible of finite type, then $V$ is generated by $V^K$ for some compact open subgroup $K$ of $G$. In this case, we set $R_c(V) = R_c(V^K)$ and $R_r(V) = R_r(V^K)$, which do not depend on the choice of $K$. If moreover $V$ is absolutely simple, we simply write $R(V)$.

\section{Scalar extension and restriction of scalars}

In this section, let $R$ be a field. 

\subsection{} Let $G_1$ and $G_2$ be locally profinite groups. We assume $G_1$ and $G_2$ contain compact open subgroups of invertible pro-order in $R$. Let $(\pi_1,V_1) \in \textup{Rep}_R(G_1)$ and set $D_1 = \textup{End}_{R[G_1]}(\pi_1)$. The representation $(\pi_1,V_1)$ is a left $D_1$-module via $f_1 \cdot v_1 = f_1(v_1)$ for $f_1 \in D_1$ and $v_1 \in V_1$. This module structure commutes with the $G_1$-action, so $V_1$ is a module over $D_1 \otimes_R R[G_1] = D_1[G_1]$. We use similar notations for $G_2$. 

We now assume $(\pi_1,V_1) \in \textup{Rep}_R(G_1)$ and $(\pi_2,V_2) \in \textup{Rep}_R(G_2)$ are two irreducible representations. Their endomorphism rings $D_1 = \textup{End}_{R[G_1]}(\pi_1)$ and $D_2 = \textup{End}_{R[G_2]}(\pi_2)$ are division algebras by Schur's lemma. We apply \cite[A VIII.210, Th 2]{bou} to obtain:

\begin{lem} \label{lem:tensor-product-of-reps-and-subreps} We set $V=V_1 \otimes_R V_2 \in \textup{Rep}_R(G_1 \times G_2)$ and $D =\textup{End}_{R[G_1 \times G_2]}(V_1 \otimes_R V_2)$. Then 
$$D  \cong D_1 \otimes_R D_2.$$ Moreover, the set $\mathbb{V}$ of subrepresentations of $V$ is in bijection with the set $\mathbb{D}$ of right sub-$D$-modules of $D$, thanks to the inclusion preserving bijection
$$\begin{array}{ccc}
\mathbb{D} & \to & \mathbb{V} \\
D' & \mapsto & D' V
\end{array}.$$
\end{lem}

\subsection{} Thanks to Lemma \ref{lem:tensor-product-of-reps-and-subreps}, which is a good replacement in the modular setting for the linear independence of characters, we are going to generalise \cite[II.4.4]{vig} which is only valid in characteristic $0$.

Let $G$ be a locally profinite group containing a compact open subgroup of invertible pro-order in $R$. We recall the definition of the action of a ``Galois'' element on a given representation. Let $R'$ be an algebraic extension of $R$ and let $w : R' \to R'$ be an automorphism of $R$-algebras. For any representation $(\rho,V) \in \textup{Rep}_{R'}(G)$, choosing an $R'$-basis $(e_i)_{i \in I}$ of $V$, we define $(w \rho , V) \in \textup{Rep}_R(G)$ by
$$[w \rho(g)]_{i,j} =  w([\rho(g)]_{i,j}]) \textup{ where } i, j \in I \textup{ and } g \in G.$$
We obtain a well-defined representation in the  sense that the isomorphism class of $(w \rho,V)$ is independent of the choice of the basis $(e_i)_{i \in I}$ \cite[II.4.1.a]{vig}.

We fix an algebraic closure $\bar{R}$ of $R$ and recall a few notions from \cite[II.4]{vig}. The rationality field $R_r(\rho)$ of a representation $\rho \in \textup{Rep}_{\bar{R}}(G)$ is defined as the fixed field of $H(\rho) = \{ w \in \textup{Gal}_R(\bar{R}) \ | \ w \rho \simeq \rho \}$ in $\bar{R}$ \textit{i.e.} 
$$R_r(\rho) = \bar{R}^{H(\rho)}.$$
Note that the extension $R_r(\rho)/R$ is not normal in general, or equivalently $H(\rho)$ is not necessarily normal in $\textup{Gal}_R(\overline{R})$, even if $R=\mathbb{Q}$. Indeed, endow $G=\mathbb{Z}$ with the discrete topology. The character $\chi$ sending $1 \in \mathbb{Z}$ to $\sqrt[3]{2} \in \overline{\mathbb{Q}}$ has rationality field $\mathbb{Q}[\sqrt[3]{2}]$ since $H(\chi) = \textup{Gal}(\overline{\mathbb{Q}}/\mathbb{Q}[\sqrt[3]{2}])$ but the extension $\mathbb{Q}[\sqrt[3]{2}]/\mathbb{Q}$ is not normal. For finite groups however, normality is ensured by the fact that there exists a cyclotomic extension, depending on the exponent of $G$, that is a splitting field \textit{i.e.} all irreducible representations can be realised over that field.

If $R$ is not perfect, we also remark that $R_r(\rho)/R$ can be really big because $R_r(\rho)$ must contain the perfect closure of $R$. To avoid complication coming from the imperfect case, we exclude it from now on.

\begin{rem} When $\rho$ is admissible and finitely generated, we can relate $R_r(\rho)$ to the rationality field defined in Section \ref{sec:rationality}. For all compact open subgroups $K$ of invertible pro-order such that $\rho^K$ generates $\rho$, we have $R_r(\rho) = R_r(\rho^K)$. We can also define the character field $R_c(\rho)$ of $rho$ by defining its trace-character as in \cite{vig}. Then $R_c(\rho) = R_c(\rho^K) = R_r(\rho^K)$ by Proposition \ref{prop:rationality-field-equal-character-field} and we can denote this field by $R(\rho)$. \end{rem}

\subsection{} From now on, assume $R$ is a perfect field. A field of realisation of $\rho$ is a field $E$ such that there exists $\tau \in \textup{Rep}_E(G)$ irreducible such that $\tau \otimes_E \bar{R} \simeq \rho$. In this case, we say that $\rho$ can be realised over $E$. Since $R$ is perfect, the algebraic extension $\bar{R}/R$ is Galois and we know \cite[II.4.1.c]{vig} that a field of realisation of $\rho$ must contain its rationality field $R_r(\rho)$ and its character field $R_c(\rho)$. If $\rho$ is irreducible and admissible, we recall that its character/rationality field $R(\rho)$ is not necessarily a field of realisation and the Schur index $m(\rho)$ measures the smallest degree of a field of realisation over $R(\rho)$. For $E$ a field extension of $R$, we let $\textup{Hom}_R(E,\bar{R})$ denote the $R$-linear embeddings of the field $E$ in $\bar{R}$.

\begin{theo} \label{thm:decomposition_scalar_extension} Let $R$ be a perfect field. Let $\pi \in \textup{Rep}_R(G)$ be irreducible and admissible. Then $D=\textup{End}_{R[G]}(\pi)$ has finite dimension over $R$. Let
\begin{itemize}[label=$\bullet$]
\item $E$ be the centre of $D$ and $n = \textup{dim}_R(E)$;
\item $m$ be the degree of $D$ over its centre $E$ \textit{i.e.} $m^2 = \textup{dim}_E(D)$.
\end{itemize}
Let $\rho \in \textup{Rep}_{\bar{R}}(G)$ be an irreducible factor in $\pi \otimes_R \bar{R}$. Let $\mathcal{O}_\rho$ be the $\textup{Gal}_R(\bar{R})$-orbit of $\rho$. Then $\mathcal{O}_\rho$ is a finite set \textit{i.e.} $R(\rho)/R$ is finite and we have
$$\pi \otimes_R \bar{R} \simeq m \ \bigg( \bigoplus_{\rho_w \in \mathcal{O}_\rho} \rho_w \bigg).$$
Moreover there exists a bijection $\mathcal{O}_\rho \simeq \textup{Hom}_R(E,\bar{R})$ compatible with the $\textup{Gal}_R(\bar{R})$-actions on each side. In particular $|\mathcal{O}_\rho| = n$ and $R(\rho) \simeq E$. \end{theo}

\begin{proof} We are going to show there exists a finite extension $R'$ of $R$ in $\bar{R}$ such that the result is true over $R'$. Our goal is to prove that the representation $\pi_1 \otimes_R R'$ is a sum of absolutely irreducible representations with the same multiplicity  and obtained as Galois conjugate of one another. First of all, $D$ is a division algebra by Schur's lemma and $R$ is contained in its centre $E$. This division algebra has finite dimension over $R$ because $\pi$ is admissible and irreducible. In particular there exists a (separable) extension $E'$ of $E$ of degree $m$, where $m^2 = \textup{dim}_E(D)$, such that $D \otimes_E E' \simeq M_{m}(E')$ is split.

We embed $E'$, and therefore $E$, in $\bar{R}$ and take the Galois closure $R'$ of $E'$ in $\bar{R}$. We see $R'$ as a an irreducible representation of the group $\mathbb{Z}$ endowed with the discrete topology in the following way. Thanks to the primitive element theorem, there exists a non-zero $\beta \in R'$ with minimal polynomial $P$ such that $R' = R[\beta] \simeq R[X]/(P(X))$. The representation $(\pi',R') \in \textup{Rep}_R(\mathbb{Z})$ defined by $\pi'( 1 ) = \beta \in \textup{GL}_R(R')$ is irreducible.

According to Lemma \ref{lem:tensor-product-of-reps-and-subreps}, the subrepresentations of $\pi_1 \otimes_R \pi_2 = \pi_1 \otimes_R R'$ correspond to right sub-$(D \otimes_R R')$-modules of $D \otimes_R R' = \textup{End}_G(\pi) \otimes_R R'$. Since
$$E \otimes_R R' \simeq \prod_{\textup{Hom}_R(E,R')} R'$$
and $|\textup{Hom}_R(E,R')| = n$, we deduce that
$$D \otimes_R R' \simeq \prod_{1 \leq k \leq n} M_m(R').$$
Moreover $\pi \otimes_R R'$ is semisimple by \cite[II.4.2]{vig}. Since $D \otimes_R R' = \textup{End}_G(\pi \otimes_R R')$ by Lemma \ref{lem:tensor-product-of-reps-and-subreps} again, there exist non-isomorphic irreducible representations $(\tau_k)$ in $\textup{Rep}_{R'}(G)$ such that
$$\pi \otimes_R R' \simeq \bigoplus_{1 \leq k \leq n} (m \tau_k).$$
Moreover $\textup{End}_G(\tau_k) = R'$, so each $\tau_k$ is absolutely irreducible. As a result $\rho_k = \tau_k \otimes_{R'} \bar{R}$ is irreducible. Hence $\pi \otimes_R \bar{R} \simeq m  (\bigoplus_{1 \leq k \leq n} \rho_k)$.

There remains to show that $\textup{Gal}_R(R')$ acts simply transitively on the $n$ isomorphism classes defined by the family $(\tau_k)$. Because $\pi$ is defined over $R$, ${}^\sigma (\pi \otimes_R R')$ and $\pi \otimes_R R'$ are isomorphic for all $\sigma \in \textup{Gal}_R(R')$. Let $e_k$ be the idempotent in $D$ which cuts the $\tau_k$-isotypic part of $\pi \otimes_R R'$. Note that $e_k \in E \otimes_R R'$ belongs to the centre of $D$. The Galois group $\textup{Gal}_R(R')$ acts transitively on $\textup{Hom}_R(E,R')$. It also induces an action on $E \otimes_R R'$ which acts transitively on the $e_k$'s. Let $\sigma \in \textup{Gal}_R(R')$ and let $k'$ be such that ${}^\sigma e_k = e_{k'}$. We obtain
$${}^\sigma(m \tau_k) \simeq {}^\sigma(e_k (\pi \otimes_R R')) \simeq e_{k'} {}^\sigma(\pi \otimes_R R') \simeq e_{k'} (\pi \otimes_R R') \simeq m \tau_{k'}.$$
So the $\tau_k$ form a single orbit under $\textup{Gal}_R(R')$. Furthermore there exists a unique $\tau$ among the $\tau_k$'s such that $E \to D \otimes_R R' \to \textup{End}_{R'[G]}(\tau) = R'$ corresponds to the natural containment $E \subset R'$. We have $\mathcal{O}_\tau \simeq \textup{Hom}_R(E,R')$ by associating to $\tau_k$ the embedding $E \to D \otimes_R R' \to \textup{End}_{R'[G]}(\tau_k) =R'$. This bijection is compatible with the natural action of $\textup{Gal}_R(R')$ on each side. Let $\sigma \in \textup{Gal}_R(R')$. Then ${}^\sigma \tau \simeq \tau$ if and only if $E \subseteq R'$ is preserved by $\sigma$ \textit{i.e.} if $i_1 : E \subseteq R'$ is the natural containment, we have $\sigma \circ i_1 = i_1$. We deduce that $E(\tau) = E$ in $R'$. This completes the proof. \end{proof}

\subsection{} We have the following lemma for the restriction of scalars:

\begin{lem} \label{lem:scalar_restriction_uniqueness} Let $R$ be a perfect field. Let $\rho \in \textup{Rep}_{\bar{R}}(G)$ be an irreducible admissible representation that can be realised over a finite extension of $R$. Then there exists an irreducible admissible representation $\pi \in \textup{Rep}_R(G)$, unique up to isomorphism, such that $\pi \otimes_R \bar{R}$ contains $\rho$ as a subquotient. We denote it by $\pi(\rho,R)$. Moreover, if $R'$ is a field of realisation of $\rho$ of minimal degree such that $\tau \otimes_{R'} \bar{R} = \rho$ with $\tau \in \textup{Rep}_{R'}(G)$, then 
$$\pi(\rho,R) \simeq \tau|_R.$$ \end{lem} 

\begin{proof} We start by proving the uniqueness statement. For all irreducible admissible $\pi \in \textup{Rep}_R(G)$, the representation $(\pi \otimes_R \bar{R})|_R$ is $\pi$-isotypic. Therefore, for any subquotient $W$ of $\pi \otimes_R \bar{R}$, the representation $W|_R$ is $\pi$-isotypic. In particular this holds when $W = \rho$. As $\rho|_R$ is isotypic, this ensures the uniqueness of the representation $\pi$ assuming it exists.

We realise $\rho$ over a finite extension $R'$ of $R$ with $\tau \otimes_{R'} \bar{R} = \rho$. It is clear that $\tau$ is admissible because $\rho$ itself is. This implies that $\tau|_R$ is admissible. Since $\tau|_R$ is of finite type, it admits an irreducible quotient $\pi$ \textit{i.e.} there exists a non-zero morphism $f : \tau|_R \to \pi$. Let $\phi : R[G] \to R'[G]$ be the obvious inclusion. Then the forgetful functor $\textup{Rep}_{R'}(G) \to \textup{Rep}_R(G)$ has the functor $\textup{Hom}_{R[G]}(R'[G],-)$ as a right adjoint. By adjunction, there corresponds to $f$ a non-zero morphism $f' : \tau \to \textup{Hom}_{R[G]}(R'[G],\pi)$ so $f'$ is injective. Hence $\tau|_R$ is a subrepresentation of
$$(\textup{Hom}_{R[G]}(R'[G],\pi))|_R \simeq \bigoplus_{[R':R]} \pi.$$
So $\tau|_R$ is $\pi$-isotypic with finite multiplicity. We deduce that $\pi$ is admissible because $\tau|_R$ is and the functor of invariants for compact open subgroups is left exact. Therefore $\pi = \pi(\rho,R)$ exists and is admissible.

There remains to show the assertion $\pi(\rho,R) \simeq \tau|_R$ if the extension $R'/R$ has minimal degree. We let $\pi$ denote $\pi(\rho,R)$ below to lighten notations. We have already shown that $\tau|_R$ is $\pi$-isotypic with finite multiplicity. So we simply need to show it is irreducible. By Theorem  \ref{thm:decomposition_scalar_extension}, there exists a subfield $E'$ of $R'$ isomorphic to the centre $E$ of $D = \textup{End}_G(\pi)$ such that $R'/E'$ has degree $m$. In particular
$$\pi \otimes_{E'} R' \simeq m \tau \textup{ and } (\pi \otimes_{E'} R')|_R = \oplus_{[R':E']} \pi = m \pi.$$
Then $m \tau|_R \simeq m \pi$ \textit{i.e.} $\tau|_R \simeq \pi$. \end{proof}

\begin{cor} \label{cor:irreducible_implies_admissible} Let $G$ be a reductive gorup over a non-archimedean local field. Suppose $R$ is a perfect field and $G$ contains open subgroups of invertible pro-order in $R$. Then an irreducible representation in $\textup{Rep}_R(G)$ is admissible. \end{cor}

\begin{proof} We first assume $G$ is connected. On the one hand, an irreducible representation $\rho$ in $\textup{Rep}_{\bar{R}}(G)$ is admissible. This well-known fact \cite[Chap. II, 2.8]{vig} is a consequence of the existence of the fact that it is true for cuspidal representations, we then deduce it for all irreducible by the existence of the cuspidal support and by other properties of parabolic induction. On the other hand, any irreducible representation $\rho$ in $\textup{Rep}_{\bar{R}}(G)$ can be realised over a finite extension of $R$ by \cite[Chap. II, 4.7]{vig}. We can apply \ref{lem:scalar_restriction_uniqueness} to conclude.

If $G$ is not connected, then its connected part is a normal subgroup of finte index. The same arguments as earlier apply by defining parabolics and cuspidality in the non-connected case. \end{proof}

\begin{rem} The corollary also applies to covering groups because they have parabolic subgroups, we can define cuspidality and we have the existence of the cuspidal support. \end{rem}

\section{Largest isotpyic quotients}

Let $R$ be a perfect field. We suppose there exists an open subgroup of $G_1 \times G_2$ of invertible pro-order in $R$.

\subsection{} The following theorem generalises the results of \cite{flath} obtained for $R=\mathbb{C}$ and \cite[Th. A.4]{vig_invent} obtained for $R$ is algebraically closed.

\begin{theo} Let $R$ be a perfect field.
\begin{enumerate}[label=\textup{\alph*)}]
\item If $\pi_1 \in \textup{Rep}_R(G_1)$ and $\pi_2 \in \textup{Rep}_R(G_2)$ are two irreducible admissible representations, then $\pi_1 \otimes_R \pi_2 \in \textup{Rep}_R(G_1 \times G_2)$ is a semisimple admissible representation.
\item If $\pi$ is an irreducible admissible representation in $\textup{Rep}_R(G_1 \times G_2)$, there exist irreducible admissible representations $\pi_1$ of $G_1$ and $\pi_2$ of $G_2$ such that $\pi$ is a quotient of $\pi_1 \otimes_R \pi_2$. Moreover $\pi_1$ and $\pi_2$ are unique up to isomorphism \textit{i.e.} $\pi$ determines these isomorphism classes. \end{enumerate} \end{theo}

\begin{proof} a) By Lemma \ref{lem:tensor-product-of-reps-and-subreps}, it is equivalent to show that $D= \textup{End}_{G_1}(\pi_1) \otimes_R \textup{End}_{G_2}(\pi_2)$ is a semisimple $R$-algebra. Since $D$ has finite dimension over $R$ by admissibility of $\pi_1$, we will simply prove that its centre is reduced. However, its centre is $E_1 \otimes_R E_2$ where $E_1$ and $E_2$ are the centres of the respective endomorphism algebras. Because $R$ is perfect, these finite extensions $E_1$ and $E_2$ are separable, which ensures $E_1 \otimes_R E_2$ is reduced.

\noindent b) Same proof as \cite[Th A.4]{vig}, which follows \cite{flath} and uses \cite[A VIII.208]{bou}. \end{proof}

In the second point above, the representation $\pi$ can happen to be a strict quotient of $\pi_1 \otimes_R \pi_2$ since the latter is not necessarily irreducible. We give a sufficient condition so that the tensor product of two irreducible representations is irreducible.

\begin{lem} \label{lem:produit_tensoriel_factorisation}Suppose the representations $\pi_1$ and $\pi_2$ are irreducible admissible and there exists a non-zero morphism of $R$-algebras $D_2 \to D_1^\textup{op}$. Then $\pi_1 \otimes_{D_2} \pi_2 \in \textup{Rep}_R(G_1 \times G_2)$ is an irreducible admissible representation. \end{lem}

\begin{proof} Let $v = \sum v_1^i \otimes_{D_2} v_2^i \in \pi_1 \otimes_{D_2} \pi_2$. Since $D_2= \textup{End}_{G_2}(\pi_2)$ is a division algebra, we can choose a finite family $v_2^i$ that is free over $D_2$ and such that $v = \sum v_1^i \otimes_{D_2} v_2^i$ with $v_1^i \neq 0$. We are going to show that $V_1 \otimes_{D_2} v_2^1$ is contained in the subrepresentation generated by $v$, which is enough to show that $v$ generates $\pi_1 \otimes_{D_2} \pi_2$. To do so, we choose an element $f_2 \in R[G_2]$ such that
$$\pi_2(f_2) v_2^i = \left\{ \begin{array}{cc} 
0 & \text{ if } i\neq 1 \\
v_2^1 & \text{ if } i=1 \end{array} \right. .$$
Such an element always exists. Indeed, let $A$ be the image of $R[G_2] \to \textup{End}_R(\pi_2)$. Then $A$ is contained in $\textup{End}_{\textup{End}_{G_2}(\pi_2)}(\pi_2) = \textup{End}_{\textup{comm}(A)}(\pi_2) = \textup{comm}(\textup{comm}(A))$ where $\textup{comm}$ means the centraliser of an algebra in $\textup{End}_R(\pi_2)$. We consider the subspace $W_2$ of $\pi_2$ generated by the $v_2^i$'s. The map we want to interpolate is the projection on $v_2^1$ with kernel generated by the other $v_2^i$'s. In particular, there exists $b \in \textup{End}_{D_2}(\pi_2) = \textup{comm}(\textup{comm}(A))$ whose restriction $b|_{W_2}$ to $W_2$ realises this projection. As the family $v_2^i$ is finite, there exists $a \in A$ such that $a|_{W_2} = b|_{W_2}$ by \cite[Chap. I, B.6]{vig}. Hence the existence of $f_2$. So $V_1 \otimes_{D_2} v_2^1$ is contained in the subrepresentation generated by $v$. \end{proof}

\subsection{} The following two lemmas generalise \cite[Chap. 2, Lem. III.3 \& Lem. III.4]{mvw}. The modifications in the proofs are very minor, by considering tensor products over division algebras instead of algebraically closed fields, so we omit the proofs.

\begin{lem} Let $(\pi_1,V_1) \in \textup{Rep}_R(G_1)$ be an irreducible admissible representaion. Let $(\pi_2,V_2) \in \textup{Rep}_R(G_2)$. Suppose $V_2$ is endowed with a structure of right $D_1$-module compatible with the action of $G_2$, this means we have a morphism of $R$-algebras
$$D_1 \to D_2^{\textup{op}}.$$
Let $V \in \textup{Rep}_R(G_1 \times G_2)$ be a subrepresentation in $V_2 \otimes_{D_1} V_1$. There exists a subrepresentation $V_2'$ of $V_2$, endowed with a structure of right $D_1$-module, such that $V = V_2' \otimes_{D_1} V_1$ in $\textup{Rep}_R(G_1 \times G_2)$. \end{lem}

When the field $R$ is algebraically closed, the results above simplify significantly because $D_1$ and $D_2$ are simply $R$. The result below generalises the usual largest isotypic quotient, which has has a simpler form again over algebraically closed fields.

\begin{lem} \label{lem:pi_coinvariant} Let $(\pi,V) \in \emph{Rep}_R(G_1 \times G_2)$. Let $(\pi_1,V_1) \in \textup{Rep}_R(G_1)$ be irreducible admissible.
\begin{itemize}[label=$\bullet$]
\item We define a subrepresentation of $V$ by
$$V[\pi_1]= \bigcap_{f \in \emph{Hom}_{G_1}(V,V_1)} \emph{Ker} (f) \in \textup{Rep}_R(G_1 \times G_2).$$
The largest $\pi_1$-isotypic quotient of $V$ is the representation
$$V_{\pi_1} = V / V[\pi_1] \in \textup{Rep}_R(G_1 \times G_2).$$ 
\item There exists an $\mathcal{H}_R(G_2)-D_1$-bimodule $(\pi_2,V_2)$, unique up to isomorphism, such that
$$V_{\pi_1} \simeq \pi_2 \otimes_{D_1} \pi_1.$$
Moreover, we have an isomorphism of $\mathcal{H}_R(G_2)-D_1$-bimodule :
$$V_2 \simeq (V \otimes_R \textup{Hom}_{D_1}(V_1,D_1)^{\infty})_{1_{G_1}}.$$ \end{itemize} \end{lem}

We can also consider several largest isotypic quotients at once. The proof of the result below is rather straightforward for two irreducible representations, and the general case follows by a simple induction argument, so the proof is left as an exercise for the reader:

\begin{cor} \label{cor:finite_family_of_pi1_coinvariants} Let $(\pi_{1,i})_{i \in I}$ be a finite family of non-isomorphic irreducible representations. For $V \in \textup{Rep}_R(G_1 \times G_2)$, we denote by $p_{\pi_{1,i}}$ the projection $V \mapsto V_{\pi_{1,i}}$. Then
$$p_I : v \in V \mapsto (p_{\pi_{1,i}}(v) )_{i\in I} \in \oplus_{i\in I} V_{\pi_1^i}$$
is surjective and has kernel
$$\bigcap_{i \in I} V[\pi_{1,i}] = \bigcap_{f \in \textup{Hom}_{G_1}(V, \oplus_{i \in I} \pi_{1,i})} \textup{Ker} (f).$$ \end{cor}

\subsection{} We fix an algebraic closure $\bar{R}$ of a perfect field $R$. For all irreducible admissible representation $\pi_1$ in $\textup{Rep}_R(G_1)$, Theorem \ref{thm:decomposition_scalar_extension} guarantees the existence of an irreducible admissible $\rho_1 \in \textup{Rep}_{\bar{R}}(G_1)$ such that
$$\pi_1 \otimes_R \bar{R} \simeq m_1 \bigg( \bigoplus_{\rho_1^\sigma \in \mathcal{O}_{\rho_1}} \rho_1^\sigma \bigg)$$
where $\mathcal{O}_{\rho_1}$ is the Galois orbit of $\rho_1$, which is a finite set in bijection with the embeddings of fields $\textup{Hom}_R(R(\rho_1),\bar{R})$. We write ${}^w \rho_1$ for the corresponding ${}^\sigma \rho \in \mathcal{O}_{\rho_1}$. 

Let $V \in \textup{Rep}_R(G_1 \times G_2)$. We relate the space of $\pi_1$-coinvariants $V_{\pi_1}$ to the spaces of $({}^w \rho_1)$-coninvariants $V_{{}^w \rho_1}$ for $w \in \textup{Hom}_R(R(\rho_1),\bar{R})$ in the result below:

\begin{theo} \label{thm:pi_coinvariants_scalar_extension} Suppose $R$ is perfect. Let $(\pi_1,V_1)$ be an irreducible admissible representation in $\textup{Rep}_R(G_1)$. We consider the decomposition of $\pi_1 \otimes_R \bar{R}$ given in Theorem \ref{thm:decomposition_scalar_extension}. Then for all $V \in \textup{Rep}_R(G_1 \times G_2)$, we have
$$V_{\pi_1} \otimes_R \bar{R} \simeq \bigoplus_{w \in \textup{Hom}_R(R(\rho_1),\bar{R})} (V \otimes_R \bar{R})_{{}^w \rho_1}.$$ \end{theo}

\begin{proof} By Corollary \ref{cor:finite_family_of_pi1_coinvariants}, the map
$$V \otimes_R \bar{R} \to \bigoplus_{w \in \textup{Hom}_R(E
R(\rho_1),\bar{R})} (V\otimes_R \bar{R})_{{}^w \rho_1}$$
is surjective and has kernel $\bigcap_w (V \otimes_R \bar{R}) [{}^w \rho_1]$. We are going to show that
$$V[\pi_1] \otimes_R \bar{R} = \bigcap_{w \in \textup{Hom}_R(R(\rho_1),\bar{R})} (V \otimes_R \bar{R})[{}^w \rho_1].$$
The latter equality will lead to $(V \otimes_R \bar{R})/ (V[\pi_1] \otimes_R \bar{R}) \simeq V_{\pi_1} \otimes_R \bar{R}$ is the largest isotypic quotient of $V$ associated to the finite family $({}^w \rho_1)_w$ in the sense of Corollary \ref{cor:finite_family_of_pi1_coinvariants}.

The direct inclusion is the easiest part. Let $v \in V[\pi_1]$. We want to prove that for all $w \in \textup{Hom}_R(R(\rho_1),\bar{R})$ and all $f \in \textup{Hom}_{\bar{R}[G_1]}(V \otimes_R \bar{R},{}^w \rho_1)$, we have $v \otimes_R 1 \in \textup{Ker}(f)$. In particular, such an $f$ defines a morphism of $R[G_1]$-modules $(V \otimes_R \bar{R})|_R \to ({}^w \rho_1)|_R$ by restriction of scalars. Moreover, the morphism $f$ is non-zero if and only if its restriction to $V\otimes_R 1 = \{ v \otimes_R 1 \ | \ v \in V\}$ is non-zero. However, according to the first lines of the proof of Lemma \ref{lem:scalar_restriction_uniqueness}, we know that the representation $({}^w \rho_1)|_R$ is $\pi_1$-isotypic. Therefore $f|_{V \otimes_R 1} : V \simeq V \otimes_R 1 \to ({}^w \rho_1)|_R \simeq \oplus \pi_1$. So $f(v \otimes_R 1) =0$ by definition of $V[\pi_1]$.

Regarding the reverse inclusion, we know thanks to Lemma \ref{lem:pi_coinvariant} that there exists a smooth $R[G_2]-D_1$-bimodule $V_2$ such that $V_{\pi_1} \simeq V_2 \otimes_{D_1} V_1$ where $D_1 = \textup{End}_{G_1}(\pi_1)$ is a division algebra of degree $m_1$ over its centre. Moreover, we have an isomorphism of representations
$$V_{\pi_1} \otimes_R \bar{R} \simeq (V_2 \otimes_R \bar{R}) \otimes_{D_1 \otimes_R \bar{R}} (V_1 \otimes_R \bar{R}).$$
Thanks to Theorem \ref{thm:decomposition_scalar_extension}, the ring $D_1 \otimes_R \bar{R}$ is isomorphic to $\prod_w e_w D \simeq \prod_w M_{m_1}(\bar{R})$ where $(e_w)_w$ is a system of primitive central idempotents in $D_1 \otimes_R \bar{R}$. We deduce that
$$V_{\pi_1} \otimes_R \bar{R} \simeq \bigoplus V_w \otimes_{e_w D} (m_1 {}^w \rho_1)$$
where $m_1 {}^w \rho_1 = e_w (\pi_1 \otimes_R \bar{R})$ and $V_w = V_2 e_w$.

We now prove there exists a representation $V_{2,w} \in \textup{Rep}_{\bar{R}}(G_2)$ such that
$$V_w \otimes_{e_w D} (m_1 ({}^w \rho_1)) \simeq V_{2,w} \otimes_{\bar{R}} ({}^w \rho_1).$$
Indeed $e_w D \simeq M_{m_1}(\bar{R})$, and by denoting $e_{i,j}$ the elementary matrix in $M_{m_1}(\bar{R})$, we have $e_{1,1} + \dots + e_{m_1,m_1} = \textup{Id}_{m_1}$ which is a decomposition of the unit as a sum of idempotents. These idempotents are not necessarily central in $M_{m_1}(\bar{R})$. Nevertheless, each $e_{i,i}$ defines a map $v \in V_w \mapsto v e_{i,i} \in V_w$ which is a morphism of $\bar{R}[G_2]$-modules. In addition $e_{i,i} V_w \simeq e_{1,1} V_w$ for all $i$. Denoting $V_{2,w} = V_w e_{1,1} \in \textup{Rep}_{\bar{R}}(G_2)$, we have $V_w \simeq m_1 V_{2,w}$ and $(m_1 V_{2,w}) \otimes_{M_{m_1}(\bar{R})} (m_1 ({}^w \rho_1)) \simeq V_{2,w} \otimes_{\bar{R}} ({}^w \rho_1)$. Therefore the quotient $V \otimes_R \bar{R} \to V_{2,w} \otimes_{\bar{R}} ({}^w \rho_1)$ factors through $(V \otimes_R \bar{R})_{{}^w \rho_1}$ by definition of the largest ${}^w \rho_1$-isotypic quotient. So the quotient $V \otimes_R \bar{R} \to V_{\pi_1} \otimes_R \bar{R} \simeq \oplus_w V_{2,w} \otimes_{\bar{R}} ({}^w \rho_1)$ factors through $\oplus_w (V \otimes_R \bar{R})_{{}^w \rho_1}$. In other words $\cap_w (V\otimes_R \bar{R})[{}^w \rho_1] \subseteq V[\pi_1] \otimes_R \bar{R}$. \end{proof}

\section{Isotypic lifts and rationality} \label{sec:isotypic_lifts_and_rationality}

Let $H_1$ and $H_2$ be locally profinite groups. Let $R$ be a perfect field and assume there exist open subgroups of $H_1 \times H_2$ of invertible pro-order in $R$. Assume all irreducible representations in $\textup{Rep}_R(H_1)$ and $\textup{Rep}_R(H_2)$ are admissible. Let $V \in \textup{Rep}_R(H_1 \times H_2)$.

\subsection{} Given an irreducible representation $\pi_1$ of $H_1$, Lemma \ref{lem:pi_coinvariant} allows us to define the isotypic lift $\Theta(\pi_1)$, which is a representation of $H_2$ endowed with a compatible right action of $D_1 = \textup{End}_{R[H_1]}(\pi_1)$, such that $V_{\pi_1} \simeq \Theta(\pi_1) \otimes_{D_1} \pi_1$. We use the notation $\Theta$ for the $\pi_1$-isotypic lift as there is an obvious analogy with the definition of the theta lifts.

\begin{theo} \label{thm:big_theta_def} Let $\pi_1 \in \textup{Irr}_R(H_1)$. There exists $\Theta(\pi_1) \in \textup{Rep}_R(H_2)$, endowed with a compatible structure of right $D_1$-module, such that $V_{\pi_1} \simeq \Theta(\pi_1) \otimes_{D_1} \pi_1$. Moreover $\Theta(\pi_1)$ is unique up to isomorphism of smooth $R[H_1]-D_1$-bimodules. \end{theo}

When $R$ is algebraically closed, or $\pi_1$ is absolutely irreducible, the division algebra $D_1$ above is simply $R$. The structure of right $D_1$-module then corresponds to the natural $R$-module structure of the representation $\Theta(\pi_1)$. When $R=\mathbb{C}$, we then find the usual definition of an isotypic lift, such as the big theta lift.

\subsection{} Let $\pi_1 \in \textup{Irr}_R(H_1)$. We are going to define the three main statements at the heart of a correspondence such as the theta correspondence. The first one is

(Fin) the $R[H_2]-D_1$-bimodule $\Theta(\pi_1)$ has finite length.

\noindent If (Fin) holds, the maximal semisimple quotient $\theta(\pi_1)$ of the $R[H_2]-D_1$-bimodule $\Theta(\pi_1)$, also called the cosocle, is well-defined. We add the second statement

(Irr) \ $\theta(\pi_1)$ is irreducible or $0$.

\noindent We also add when (Fin) and (Irr) hold for all $\pi_1$, the third statement

(Uni) $0 \neq \theta(\pi_1) \simeq \theta(\pi_1')$ if and only if $\pi_1 \simeq \pi_1'$.

\begin{rem} In particular $0 \neq \theta(\pi_1) \simeq \theta(\pi_1')$ means an isomorphism of $R[H_2]$-module that is compatible with the respective structures of right modules via some isomorphism $\textup{End}_{R[H_1]}(\pi_1) \simeq \textup{End}_{R[H_1]}(\pi_1')$. \end{rem}

These statements are related to the field $R$ and we want to show these three statements over $R$ are equivalent to the analogous statements over an algebraic closure $\bar{R}$ of $R$ using $V \otimes_R \bar{R}$. We call ($\Theta_R$) the statements (Fin)-(Irr)-(Uni) over $R$. The goal of the section is to prove the following:

\begin{theo} \label{thm:theta_corresp_over_R_and_bar_R} Recall that $R$ is a perfect field and fix an algebraic closure $\bar{R}$ of $R$. Then the following assertions are equivalent
\begin{enumerate}[label=\textup{\alph*)}]
\item \textup{($\Theta_R$)} hold for all $\pi_1 \in \textup{Irr}_R(H_1)$;
\item \textup{($\Theta_{\bar{R}}$)} hold for all $\rho_1 \in \textup{Irr}_{\bar{R}}(H_1)$.
\end{enumerate} \end{theo}

The proof of the theorem will occupy the next paragraphs. We prove the equivalence of the three statements separately.

\subsection{} We start by explaining how the isotypic lift behaves with respect to scalar extension. Let $\pi_1 \in \textup{Irr}_R(H_1)$. By choosing an irreducible factor $\rho_1$ in $\pi_1 \otimes_R \bar{R}$, we can consider the decomposition of Theorem \ref{thm:decomposition_scalar_extension}
$$\pi_1 \otimes_R \bar{R} \simeq  m_1 \bigg(\bigoplus_{w \in \textup{Hom}_R(R(\rho_1),\bar{R})} {}^w \rho_1 \bigg).$$
Moreover $E_1 \simeq R(\rho_1)$ where $E_1$ is the centre of $D_1$.

\begin{lem} \label{lem:big_theta_scalar_extension} We have an isomorphism of representations in $\textup{Rep}_{\bar{R}}(H_1 \times H_2)$
$$(\Theta(\pi_1) \otimes_{D_1} \pi_1) \otimes_R \bar{R} \simeq \bigoplus_{w \in \textup{Hom}_R(R(\rho_1),\bar{R})} \Theta({}^w \rho_1) \otimes_{\bar{R}} {}^w \rho_1.$$
Moreover, by considering $\Theta(\pi_1)$ as a representation in $\textup{Rep}_{E_1}(H_2)$, there exists a bijection $\varphi : \textup{Hom}_R(R(\rho_1), \bar{R}) \to \textup{Hom}_R(E_1,\bar{R})$ such that
$$\Theta(\pi_1) \otimes_{E_1,\varphi(w)} \bar{R} \simeq  m_1 \Theta({}^w \rho_1)$$
as $R[H_2]-\mathcal{M}_{m_1}(\bar{R})$-modules via $D_1 \otimes_{E_1, \varphi(w)} \bar{R} \simeq \mathcal{M}_{m_1}(\bar{R})$. \end{lem}

\begin{proof} Theorem \ref{thm:pi_coinvariants_scalar_extension} ensures that
$$V_{\pi_1} \otimes_R \bar{R} \simeq \bigoplus_{w \in \textup{Hom}_R(R(\rho_1),\bar{R})} (V \otimes_R \bar{R})_{{}^w \rho_1}.$$
Now we can use Theorem \ref{thm:big_theta_def} to obtain the first isomorphism of the lemma.

Any representation $V_{\pi_1} \otimes_{E_1,w'} \bar{R}$ with $w' \in \textup{Hom}_R(E_1, \bar{R})$ is isomorphic to one and only one of $(V \otimes_R \bar{R})_{{}^w \rho_1}$ with $w \in \textup{Hom}_R(R(\rho_1),\bar{R})$. This defines the bijection $\varphi$.

To end the proof, the isomorphism $(\Theta(\pi_1) \otimes_{D_1} \pi_1) \otimes_{E_1,\varphi(w)} \bar{R} \simeq \Theta({}^w \rho_1) \otimes_{\bar{R}} {}^w \rho_1$ induces the desired isomorphism thanks to the fact that $D_1' = D_1 \otimes_{E_1,\varphi(w)} \otimes \bar{R} \simeq \mathcal{M}_{m_1}(\bar{R})$ and
\begin{eqnarray*}
(\Theta(\pi_1) \otimes_{D_1} \pi_1) \otimes_{E_1,\varphi(w)} \bar{R}  & \simeq &  (\Theta(\pi_1) \otimes_{E_1,\varphi(w)} \bar{R}) \otimes_{D_1'} (\pi_1 \otimes_{E_1,\varphi(w)} \bar{R}) \\
  & \simeq & (\Theta(\pi_1) \otimes_{E_1,\varphi(w)} \bar{R}) \otimes_{D_1'} (m_1 {}^w\rho_1) \\
   & \simeq & \Theta({}^w \rho_1) \otimes_{\bar{R}} {}^w \rho_1.
\end{eqnarray*} 
By Morita equivalence of $\mathcal{M}_{m_1}(\bar{R})$ and $\bar{R}$, we obtain $\Theta(\pi_1) \otimes_{E_1,\varphi(w)} \bar{R} \simeq m_1 \Theta({}^w\rho_1)$. \end{proof}

\subsection{} The three propositions below prove the equivalence for each statement.

\begin{prop} The following assertions are equivlent
\begin{enumerate}[label=\textup{\alph*)}]
\item the $R[H_2]-D_1$-bimodule $\Theta(\pi_1)$ has finite length;
\item all representations $\Theta({}^w \rho_1)$ have finite length;
\item there exists ${}^w \rho_1$ such that $\Theta({}^w \rho_1)$ has finite length.
\end{enumerate} \end{prop}

\begin{proof} We only need to show that a) $\Rightarrow$ b) and c) $\Rightarrow$ a), since b) $\Rightarrow$ c) is obvious.

For a) $\Rightarrow$ b), we first recall that irreducible representations are admissible thanks to Corollary \ref{cor:irreducible_implies_admissible} and the remark thereafter. Therefore, if the $R[H_2]-D_1$-bimodule $\Theta(\pi_1)$ has finite length, Lemma \ref{lem:produit_tensoriel_factorisation} implies that $\Theta(\pi_1) \otimes_{D_1} \pi_1$ is an admissible representation of finite length in $\textup{Rep}_R(H_1 \times H_2)$. By extending scalars to $\bar{R}$ as in Theorem \ref{thm:decomposition_scalar_extension}, the representation $(\Theta(\pi_1)\otimes_{D_1} \pi_1 ) \otimes_R \bar{R}$ has finite length in $\textup{Rep}_{\bar{R}}(H_1 \times H_2)$. Then Lemma \ref{lem:big_theta_scalar_extension} allows us to conclude that $\Theta({}^w \rho_1)$ has finite length for all ${}^w \rho_1$.

For the last implication c) $\Rightarrow$ a), we use Lemma \ref{lem:big_theta_scalar_extension} again. The isomorphism
$$\Theta(\pi_1) \otimes_{E_1,\varphi(w)} \bar{R} \simeq m_1 \Theta({}^w \rho_1)$$
is compatible with the right action of $D_1 \otimes_{E_1,\varphi(w)} \bar{R} \simeq \mathcal{M}_{m_1}(\bar{R})$, therefore the $\bar{R}[H_2]-\mathcal{M}_{m_1}(\bar{R})$-bimodule $\Theta(\pi_1) \otimes_{E_1,\varphi(w)} \bar{R}$ has finite length provided $\Theta(\rho_1^w)$ has finite length in $\textup{Rep}_{\bar{R}}(H_2)$. This implies that $\Theta(\pi_1)$ is an $R[H_2]-D_1$-bimodule of finite length. \end{proof}

\begin{prop} Suppose $\Theta(\pi_1)$ is an $R[H_2]-D_1$-bimodule of finite length and denote by $\theta(\pi_1)$ its co-socle. When $\Theta(\pi_1) \neq 0$, the following assertions are equivalent
\begin{enumerate}[label=\textup{\alph*)}]
\item the $R[H_2]-D_1$-bimodule $\theta(\pi_1)$ is irreducible;
\item all representations $\theta({}^w \rho_1)$ are irreducible;
\item there exists ${}^w \rho_1$ such that $\theta(w \rho_1)$ is irreducible.
\end{enumerate} \end{prop}

\begin{proof} As in the previous proof, the implication b) $\Rightarrow$ c) is obvious. 

For a) $\Rightarrow$ b), we prove the contraposition. We then suppose there exists ${}^w \rho_1$ such that $\theta({}^w \rho_1)$ is not irreducible. We want to show there exist two irreducible representations $\tau_2$ and $\tau_2'$ in $\textup{Rep}_{\bar{R}}^{\textup{gen}}(H_2)$ such that $\Theta({}^w \rho_1) \otimes_{\bar{R}} ({}^w \rho_1)$ admits $(\tau_2 \otimes_{\bar{R}} ({}^w \rho_1)) \oplus (\tau_2' \otimes_{\bar{R}} ({}^w \rho_1))$ as a quotient.

On the one hand, the representation $(\tau_2 \otimes_{\bar{R}} ({}^w \rho_1))|_R$ of $H_1$ is $\pi_1$-isotypic. According to Lemme \ref{lem:pi_coinvariant}, there exists an $R[H_2]-D_1$-bimodule $\sigma_2$ such that $(\tau_2 \otimes_{\bar{R}} ({}^w \rho_1))|_R \simeq \sigma_2 \otimes_{D_1} \pi_2$. Moreover, it is semisimple as a representation in $\textup{Rep}_R(H_1 \times H_2)$. So $\sigma_2$ is isotypic and we let $\pi_2$ be any irreducible factor of $\sigma_2$. We use similar notations for $\tau_2'$ and $\sigma_2'$ and $\pi_2'$.

On the other hand, Lemma \ref{lem:big_theta_scalar_extension} implies $\Theta({}^w \rho_1) \otimes_{\bar{R}} ({}^w \rho_1) \simeq (\Theta(\pi_1) \otimes_{D_1} \pi_1) \otimes_{E_1,\varphi(w)} \bar{R}$. We deduce that the obvious morphism of $\Theta(\pi_1) \otimes_{D_1} \pi_1$ in the right-hand side guarantees that $\Theta(\pi_1) \otimes_{D_1} \pi_1$ admits $(\pi_2 \otimes_{D_1} \pi_1) \oplus (\pi_2' \otimes_{D_1} \pi_1)$ as a quotient. The kernel of this quotient map is of the form $\sigma_2'' \otimes_{D_1} \pi_1$ where $\sigma_2''$ is a sub-$R[H_2]-D_1$-bimodule of $\Theta(\pi_1)$ by Lemma \ref{lem:pi_coinvariant}. As a result, this quotient map induces a quotient $\Theta(\pi_1) \to \pi_2 \oplus \pi_2'$ of $R[H_2]-D_1$-bimodules whose kernel is precisely $\sigma_2''$. Hence $\theta(\pi_1)$ is not irreducible.

Finally, the implication c) $\Rightarrow$ a), we will prove the contraposition again. We then suppose that $\theta(\pi_1)$ is not irreducible. We need to show that for all ${}^w \rho_1$, the representation $\theta({}^w \rho_1)$ is not irreducible. However $\Theta(\pi_1) \otimes_{D_1} \pi_1$ admits $\theta(\pi_1) \otimes_{D_1} \pi_1$ as a quotient. As a result $(\Theta(\pi_1) \otimes_{D_1} \pi_1)\otimes_R \bar{R}$ admits $(\theta(\pi_1) \otimes_{D_1} \pi_1) \otimes_R \bar{R}$ as a quotient. But $(\Theta(\pi_1) \otimes_{D_1} \pi_1)\otimes_{E_1,\varphi(w)} \bar{R}$ admits $(\theta(\pi_1) \otimes_{D_1} \pi_1) \otimes_{E_1,\varphi(w)} \pi_1$ and $\theta(\pi_1) \otimes_{E_1,\varphi(w)} \bar{R} = m_1 \theta({}^w \rho_1)$ as a quotient. Hence $\theta({}^w \rho_1)$ is not irreducible because $\theta(\pi_1)$ is not and the scalar extension functor is exact, therefore $\theta(\pi_1) \otimes_{E_1,\varphi(w)} \bar{R}$ has length at least $2 m_1$. \end{proof}

\begin{prop} Suppose $\Theta(\pi_1)$ and $\Theta(\pi_1')$ are two $R[H_2]-D_1$-bimodules of finite length whose respective co-socles $\theta(\pi_1)$ and $\theta(\pi_1')$ are irreducible. Let $\rho_1$ and $\rho_1'$ be irreducible representations contained in the scalar extension to $\bar{R}$ of $\pi_1$ and $\pi_1'$. Assume $D_1 = \textup{End}_{R[H_1]}(\pi_1) \simeq \textup{End}_{R[H_1]}(\pi_1')$. Then the following assertions are equivalent
\begin{enumerate}[label=\textup{\alph*)}]
\item $\theta(\pi_1) \simeq \theta(\pi_1')$ as $R[H_2]-D_1$-bimodules;
\item for all ${}^w \rho_1$, there exists ${}^{w'} \rho_1'$ such that $\theta({}^w \rho_1) \simeq \theta({}^{w'} \rho_1')$;
\item there exists ${}^w \rho_1$ and ${}^{w'} \rho_1'$ such that $\theta({}^w \rho_1) \simeq \theta({}^w \rho_1')$.
\end{enumerate} \end{prop}

\begin{proof} The implication b) $\Rightarrow$ c) is still obvious.

For a) $\Rightarrow$ b), the isomorphism $\theta(\pi_1) \otimes_{E_1,\varphi(w)} \bar{R} \simeq m_1 \theta({}^w \rho_1)$ is an isomorphism of $\bar{R}[H_2]-(D_1 \otimes_{E_1,\varphi(w)} \bar{R})$-bimodules. And likewise for $\theta(\pi_1')$ and some $\theta({}^{w'}\rho_1')$. If $\theta(\pi_1) \simeq \theta(\pi_1')$, then $\theta({}^w \rho_1) \simeq \theta({}^{w'} \rho_1')$ in $\textup{Rep}_{\bar{R}}(H_2)$.

Regarding the last implication c) $\Rightarrow$ a), there exists thanks to the previous paragraph an isomorphism $\theta(\pi_1) \otimes_{E_1,\varphi(w)} \bar{R} \simeq \theta(\pi_1') \otimes_{E_1,\varphi'(w')} \bar{R}$ because $\theta({}^w \rho_1) \simeq \theta({}^{w'} \rho_1')$. Let $\sigma \in \textup{Gal}_R(\bar{R})$. Then ${}^{\sigma}(\theta(\pi_1) \otimes_{E_1,\varphi(w)} \bar{R}) \simeq \theta(\pi_1) \otimes_{E_1 , \varphi(\sigma w)} \bar{R}$. In particular, this implies that $\theta({}^{\sigma w} \rho_1) \simeq \theta({}^{\sigma w'} \rho_1)$. Therefore, there exists a bijection $\psi$ of $\textup{Hom}_R(R(\rho_1),\bar{R})$ such that for all $w_0 \in \textup{Hom}_R(R(\rho_1),\bar{R})$, we have $\theta({}^{w_0} \rho_1) \simeq \theta({}^{\psi(w_0)} \rho_1')$. We thus have isomorphisms of $\bar{R}[H_2]-(D_1\otimes_R \bar{R})$-bimodules
$$\theta(\pi_1) \otimes_R \bar{R} \simeq \oplus_{w_0} \theta(\pi_1) \otimes_{E_1 , \varphi(w_0)} \bar{R} \simeq \theta(\pi_1') \otimes_R \bar{R}.$$
By restriction of scalars, $(\theta(\pi_1) \otimes_R \bar{R})|_R$ is an $R[H_2]-D_1$-bimodule which is at the same time $\theta(\pi_1)$-isotypic and $\theta(\pi_1')$-isotypic. Therefore $\theta(\pi_1) \simeq \theta(\pi_1')$. \end{proof}

\subsection{} We now assume $R$ is algebraically closed. Let $V \in \textup{Rep}_R(H_1 \times H_2)$, we can define $\Theta_V$ as earlier. We add the subscript $V$ here because we want to be able to consider $\Theta_{V'}$ for another $V' \in \textup{Rep}_R(H_1 \times H_2)$. We have a compatibility of the isotypic lifts with the Galois action in the following sense:

\begin{theo} \label{thm:theta_lifts_and_galois_action} The isotypic lifts are compatible with the action of $\textup{Gal}_{R_0}(R)$ in the sense that, for all $\sigma \in \textup{Gal}_{R_0}(R)$ and for all $\pi_1 \in \textup{Irr}_R(H_1)$, we have
$${}^\sigma \Theta_V(\pi_1) \simeq \Theta_{{}^\sigma V}({}^\sigma \pi_1).$$ \end{theo}

\begin{proof} The morphism $V \twoheadrightarrow V_{\pi_1}$ induces ${}^\sigma V \twoheadrightarrow ({}^\sigma V)_{{}^\sigma \pi_1}$ and ${}^\sigma (V_{\pi_1}) \simeq ({}^\sigma V)_{{}^\sigma \pi_1}$. Therefore ${}^\sigma(\Theta_V(\pi_1) \otimes_R \pi_1) \simeq {}^\sigma \Theta_V(\pi_1) \otimes_R {}^\sigma \pi_1 \simeq \Theta_{{}^\sigma V}({}^\sigma \pi_1) \otimes_R {}^\sigma \pi_1$, so ${}^\sigma \Theta_V(\pi_1) \simeq \Theta_{{}^\sigma V}({}^\sigma \pi_1)$ by Theorem \ref{thm:big_theta_def} and the uniqueness of the lift. \end{proof}

\part{Galois descent on the Weil representation}

\section{Galois descent as Morita equivalences}

The Galois descent theorems -- obtained by taking the fixed points under the action of some Galois group -- can be seen as a particular case of faithfully flat descent. We give here another interpretation of the Galois descent theorems in terms of Morita equivalences, which is a representation theoretic approach of these results.

\subsection{} Let $A$ be a Dedekind ring and let $A \to B$ be a (finite) étale morphism. Let $G = \textup{Aut}_A(B)$ be the finite group of ring automorphisms of $B$ that are $A$-linear.

\begin{defi} The twisted group algebra $B'[G]$ is the $A$-algebra on the free $B$-module $\bigoplus_{g \in G} B \cdot g$ of basis $G$ endowed with the twisted multiplication $(\alpha \cdot g) \times ( \beta \cdot g') = \alpha g(\beta) \cdot g g'$.\end{defi}

A module $V$ over $B'[G]$ is equivalently a $B$-module endowed with a semi-linear action of $G$. We have the following equivalence of categories:

\begin{theo} The category of $A$-modules is equivalent to the category of $B'[G]$-modules. Given a $B'[G]$-module $V$, the natural map $V^G \otimes_A B \to V$ is an isomorphism. \end{theo}

This theorem can be seen as a consequence of Morita equivalences. Indeed, since $A \to B$ is étale, the module $B$ is a progenerator of the category of $A$-modules, therefore:
$$V \mapsto V \otimes_A B$$
is a Morita equivalence between the categories of $A$-modules and $\textup{End}_A(B)$-modules. However, the natural map:
$$B'[G] \to \textup{End}_A(B)$$
is an isomorphism. Indeed, it becomes an isomorphism at the generic fibre, this gives the injectivity. We claim the map is surjective because no prime in $A$ ramifies in $B$. It is easy to see the isomorphism holds when $B$ is free over $A$ as the determinant of the image of $g \in G$ in $\textup{End}_A(B)$ divides the discriminant, which is invertible by our ramification hypothesis. In general $B$ is locally free, so we reduce the claim to the free case by localisation.

We need to explain why $V^G \otimes_A B \to V$ is an isomorphism. As a consequence of the previous equivalence of categories, there exists an\ $A$-submodule $W$ of $V$ such that $V = W \otimes_A B$. But $(W \otimes_A B)^G = W$, so we deduce that $W = V^G$.

\subsection{} We can generalise the Galois descent theorem above when $A \to B$ is proétale \textit{i.e.} when we can write $B$ as a limit of $B_i$'s where $A \to B_i$ is étale. It is no longer true that the group $G = \textup{Aut}_A(B)$ is finite, but it is profinite. If we set $G_i=\textup{Aut}_A(B_i)$, we have:
$$G = \underset{\leftarrow}{\textup{lim}} \ G_i.$$
The profinite case almost works in the same way as the finite case, except that the action must now be coming from finite groups \textit{i.e.} be smooth in the usual sense. In particular $V$ is smooth if and only $V = \cup_i V^{G_i}$. We deduce easily from the finite case that:

\begin{theo} \label{thm:pro-etale-descent} The category of $A$-modules is equivalent to the category of $B$-modules with semi-linear smooth action of $G$. Given a smooth $B'[G]$-module $V$, the natural map $V^G \otimes_A B \to V$ is an isomorphism. \end{theo}

\section{Preliminaries on the Weil representation}

Let $F$ be a field of characteristic not $2$, that is either a finite field of cardinality $q$ or a non-archimedean local field of residual cardinality $q$. We write $q=p^f$ where $p$ is a prime number. Let $(W, \langle \ , \ \rangle)$ be a symplectic vector space of finite dimension $n=2m$ over $F$. Let $H$ be the Heisenberg group \textit{i.e.} $H = W \times F$ as a set with group law
$$(w,t) \cdot (w',t') = (w+w',t+t'+\frac{1}{2} \langle w,w' \rangle)$$
for $w, w' \in W$ and $t, t' \in F$. The centre of $H$ is identified with $F$ via $t \mapsto (0,t)$.

Let $R_0$ be $\mathbb{Q}$ or $\mathbb{F}_\ell$ with $\ell \neq p$ and let $\mathcal{K}$ be the field extension $R_0[\zeta_p]$ if $\textup{char}(F) > 0$ and $R_0[\zeta_{p^\infty}]$ if $\textup{char}(F)=0$. The Galois group $G = \textup{Gal}(\mathcal{K}/R_0)$ is abelian. It is a subgroup of $(\mathbb{Z}/p \mathbb{Z})^\times$ if $\textup{char}(F) > 0$ and an open subgroup of $\mathbb{Z}_p^\times$ otherwise.

Let $\psi : F \to \mathcal{K}^\times$ be a non-trivial smooth character. If $\sigma \in G$, we define a non-trivial character ${}^\sigma \psi : F \to \mathcal{K}^\times$ via $\psi^\sigma(t) = \sigma( \psi(t))$. If $\gamma \in F^\times$, we define a non-trivial character $\psi^\gamma : F \to \mathcal{K}^\times$ via $\psi^\gamma(t) = \psi(\gamma t)$. Since the action of $F^\times$ on non-trivial characters is simply transitive, we deduce that for each $\{ (\sigma,\gamma) \in G \times F^\times \ | \ \psi^\sigma = \psi^\gamma \}$ is the graph of an homeomorphism which identifies $G$ with a compact subgroup of $F^\times$. 

Let $W = X \oplus Y$ be a complete polarisation. In particular $X$ and $Y$ are Lagrangians. It also defines a duality between $Y$ and $X$ by identifying $Y$ and $X^*$ via $y \mapsto \langle y , - \rangle$.

\subsection{} The Schr\"odinger model of the Heisenberg representation associated to $\psi$ and $X$ is the representation $(\rho_{\psi,X},S_{\psi,X}) \in \textup{Rep}_\mathcal{K}(H)$ with central character $\psi$ defined by
$$S_{\psi,X} = \textup{ind}_{X_H}^H(\psi_X)$$
where $X_H = X \times F$ is a subgroup of $H$ and $\psi_X : (x,t) \in X_H \mapsto \psi(t) \in \mathcal{K}^\times$ is a character. Thanks to our complete polarisation, we have a canonical isomorphism between $S_{\psi,X}$ and $C_c^\infty(Y,\mathcal{K})$ via $f \mapsto f|_Y$. This allows us to consider $S_{\psi,X}$ on the space $C_c^\infty(Y,\mathcal{K})$, which is a vector space independent of $\psi$, though the action of $H$ on this vector space will depend on $\psi$. To be more explicit
$$\rho_{\psi,X}(x+y,t) \cdot f (y') = \psi(t) \psi(\langle y' , x \rangle + \frac{1}{2} \langle y , x \rangle) f(y'+y)$$
for $x \in X$, $y \in Y$, $t \in F$ and $f \in C_c^\infty(Y,\mathcal{K})$.

For $\sigma \in G$, we define $\sigma \cdot f \in C_c^\infty(Y,\mathcal{K})$ via $(\sigma \cdot f) (y) = \sigma(f(y))$. The morphism $f \mapsto \sigma \cdot f$ defines a semilinear isomorphism $(\rho_{\psi,X},S_{\psi,X}) \simeq (\rho_{\psi^\sigma,X},S_{\psi^\sigma,X})$ of representations \textit{i.e.}
$$\sigma \cdot (\rho_{\psi,X}(h) f) = \rho_{\psi^\sigma,X}(h) (\sigma \cdot f)$$
for all $h \in H$ and all $f \in C_c^\infty(Y,\mathcal{K})$. In particular ${}^\sigma \rho_{\psi,X} \simeq \rho_{\psi^\sigma,X}$.

\subsection{} Thanks to the Stone-von Neumann Theorem \cite{mvw,trias_modular_weil}, there exists a unique $\rho_\psi \in \textup{Irr}_\mathcal{K}(H)$ with central character $\psi$. Moreover $\rho_\psi$ is admissible and absolutely irreducible. Let $\textup{tr}_{\rho_\psi} : \mathcal{H}_\mathcal{K}(H) \to \mathcal{K}$ be the trace-character of $\rho_\psi$ defined in Section \ref{sec:trace_of_admissible_reps} and recall that $R(\rho_\psi)$ is its rationality/character field and $m(\rho_\psi)$ its Schur index.

\begin{prop}\label{prop:heisenberg-rep-realisation-rationality-character} We have $R(\rho_\psi) = \mathcal{K}$ and $m(\rho_\psi) = 1$. \end{prop}

\begin{proof} Let $\sigma \in \textup{Gal}(\mathcal{K}/R_0)$. Since the representation ${}^\sigma \rho_\psi$ is irreducible and has central character $\psi^\sigma$, it is isomorphic to $\rho_{\psi^\sigma}$ by Stone-von Neumann Theorem. Then ${}^\sigma \rho_\psi \simeq \rho_\psi$ only if $\psi^\sigma = \psi$ \textit{i.e.} $\sigma = \textup{id}$. So $H(\rho_\psi) = \{ \textup{id} \}$ and $R(\rho_\psi)= \mathcal{K}$. Since $\rho_\psi$ is already realised over $R(\rho_\psi)$, we obtain $m(\rho_\psi) = 1$. \end{proof}

\subsection{} Let $\textup{Mp}(W) = \textup{Sp}(W) \times_{c_X} \{ \pm 1 \}$ be the metaplectic group and let $(\omega_{\psi,X},S_{\psi,X}) \in \textup{Rep}_{\mathcal{K}}(\textup{Mp}(W))$ be the Weil representation associated to $\psi$ and $X$ \cite{trias_modular_weil}. As earlier, we realise the Weil representations $\omega_{\psi,X}$ for each character $\psi$ in a uniform way via the isomorphism $S_{\psi,X} \simeq C_c^\infty(Y,\mathcal{K})$ of vector spaces.

From now on, we will write $(\omega_{\psi,X},C_c^\infty(Y,\mathcal{K}))$ for the Weil representation associated to $\psi$ and $X$. Let $P(X)$ be the stabiliser of $X$ in $\textup{Sp}(W)$, also called the Siegel parabolic, and denote by $P(X) = M(X) N(X)$ its Levi decomposition with respect to the complete polarisation $W = X \oplus Y$. We have an antiisomorphism $a \in \textup{GL}_F(X) \mapsto a^* \in \textup{GL}_F(Y)$ and an isomorphism $b \in \textup{Hom}_F(Y,X) \mapsto b^* \in \textup{Hom}_F(Y,X)$ via $Y \cong X^*$ and $X^* \cong Y$. We let $\textup{Hom}_F^{*,-}(Y,X) = \{ b \in \textup{Hom}_F(Y,X) \ | \ b^* = - b\}$ be the antisymmetric morphisms.

We can describe $\omega_{\psi,X}$ on generators of $\textup{Sp}(W)$ via
\begin{itemize}
\item if $m	 \in M(X)$ with $m|_X = a \in \textup{GL}_F(X)$, then $$(\omega_{\psi,X}(m,1) \cdot f) (y) = \Omega_{1,\textup{det}(a)}^\psi f(a^* y).$$
\item if $n \in N(X)$ with $n-\textup{id}_W = b \in \textup{Hom}_F^{*,-}(Y,X)$, then $$(\omega_{\psi,X}(n,1) \cdot f)(y) = \psi(\frac{1}{2} \langle b y , y \rangle ) f(y).$$
\item if $w \in \textup{Sp}(W)$ with $w(X) = Y$ and $w|_X = c \in \textup{Iso}_F(X,Y)$, then $$(\omega_{\psi,X}(w,1) \cdot f)(y) = \int_X \psi(\langle x , y \rangle) f(c x) d\mu_w^\psi(x)$$ where $\mu_w^\psi = \Omega_\mu(\psi \circ Q_w)^{-1} \mu$ and $Q_w(x) = \frac{1}{2} \langle w x , x \rangle = \frac{1}{2} \langle c x , x \rangle$ (see \cite{trias_modular_weil}). 
\end{itemize}

\subsection{} \label{sec:weil_rep_dependance_on_psi} Let $\sigma \in G = \textup{Gal}(\mathcal{K}/R_0)$ and consider the morphism $f \mapsto \sigma \cdot f$ as before. The non-normalised Weil factor sastisfies $\sigma(\Omega_{\mu}(\psi \circ Q)) = \Omega_{\mu^\sigma}(\psi^\sigma \circ Q)$. In particular $\sigma(\Omega_{1,\gamma}^\psi) = \Omega_{1,\gamma}^{\psi^\sigma}$ for all $\gamma \in F^\times$. Therefore, for all $f \in C_c^\infty(Y,\mathcal{K})$ and $(g,\lambda) \in \textup{Mp}(W)$, we have
$$\sigma \cdot (\omega_{\psi,X}(g,\lambda) f) = \omega_{\psi^\sigma,X}(g,\lambda) (\sigma \cdot f)$$
because the metaplectic cocycle $c_X$ takes values in $\{ \pm 1 \}$. In particular ${}^\sigma \omega_{\psi,X} \simeq \omega_{\psi^\sigma,X}$.

\subsection{} For $\gamma \in F^\times$, it is well-known that $\omega_{\psi^\gamma,X} \simeq \omega_{\psi,X}$ if and only $\gamma \in F^{\times 2}$. The proof is the same as in the complex case, where the if part comes from the action of $\textup{GSp}(W)$ by conjugation \cite[9.1.2]{theta_book} and the only if part from a computation of twisted Jacquet functors \cite[9.4.3]{theta_book}. We need to be more precise than that to perform our Galois descent. Let $( \ , \ )_F$ be the quadratic Hilbert symbol, which is trivial if $F$ is a finite field.

\begin{prop} \label{prop:formulas-weil-rep-character-gamma} For $\gamma \in F^\times$, we have the following identities: 
\begin{eqnarray}
\omega_{\psi^\gamma,X}(m(a),1) = (\gamma,\textup{det}(a))_F \omega_{\psi,X}(m(a),1) \\
\omega_{\psi^{\gamma},X}(n(b),1) = \omega_{\psi,X}(n(\gamma b),1) \\
\omega_{\psi^\gamma,X}(w(c),1) = \Omega_{w,\gamma} \omega_{\psi,X}(w(\gamma c),1)
\end{eqnarray} where $\displaystyle \Omega_{w,\gamma} = \frac{\Omega_\mu(\psi \circ Q_w)}{\Omega_\mu(\psi^\gamma \circ Q_w)} |\gamma|_F^m$. \end{prop}

\begin{proof} We have $\Omega_{1,\textup{det}(a)}^{\psi^\gamma} = (\gamma,\textup{det}(a))_F \Omega_{1,\textup{det}(a)}^\psi$ by the properties of the non-normalised Weil factor according to \cite[Sec 4]{trias_modular_weil}, so this proves the first formula. The second one is very straightforward. The last one comes from a simple change of variables. \end{proof}

\begin{cor} \label{cor:identification_omega_psi_and_omega_psi_gamma2} Let $\gamma \in \mathcal{O}_F^\times$ and write $m_\gamma$ for $m(\gamma \textup{id}_X)$. We have $$ \omega_{\psi^{\gamma^2},X}(g,\lambda) = \omega_{\psi,X}^{m_\gamma}(g,\lambda).$$ \end{cor}

\begin{proof} In the metaplectic group, we have $(m_\gamma,\mu) (g,\lambda) (m_\gamma,\mu)^{-1} = (m_\gamma g m_\gamma^{-1},\lambda)$. Recall that the quadratic Hilbert symbol $( \ , \ )_F$ is trivial on squares and that the non-normalised Weil factor satisfies $\Omega_\mu(\psi^{\gamma^2} \circ Q_w) = \Omega_\mu(\psi \circ Q_w \circ \gamma) = |\gamma|_F^m \Omega_\mu(\psi \circ Q_w)$ when $\gamma \in F^\times$. Therefore, when $\gamma \in \mathcal{O}_F^\times$, we obtain from Proposition \ref{prop:formulas-weil-rep-character-gamma} three equalities
\begin{eqnarray}
\omega_{\psi^{\gamma^2},X}(m,1) = \omega_{\psi,X}(m,1) = \omega_{\psi,X}(m_\gamma m m_\gamma^{-1},1) \\
\omega_{\psi^{\gamma^2},X}(n,1) = \omega_{\psi,X}(m_\gamma n m_\gamma^{-1},1) \\
\omega_{\psi^{\gamma^2},X}(w,1) = \omega_{\psi,X}(m_\gamma w m_\gamma^{-1},1)
\end{eqnarray}
Since the symplectic group $\textup{Sp}(W)$ is generated by elements of the form $m, n, w$, the genuine representations $\omega_{\psi^{\gamma^2},X}$ and $\omega_{\psi,X}^{m_\gamma}$ of $\textup{Mp}(W)$ must be equal. \end{proof}

\section{Descent when $p \neq 2$} \label{sec:descent-when-p-not-2}

\subsection{} Write $G = G_2 \times G_{2'}$ where $G_2$ is a finite $2$-group and $G_{2'}$ has pro-order prime-to-$2$. In particular the square morphism $\sigma \mapsto \sigma^2$ induces a group automorphism of $G_{2'}$. Let $\mathcal{L} = \mathcal{K}^{G_{2'}}$ be the fixed field of $G_{2'}$, which is a finite extension of $R_0$. We can always operate a Galois descent to this subfield:

\begin{theo} \label{thm:weil-representation-descent-p-not-2} The Weil representation can be realised over $\mathcal{L}$. \end{theo}

\begin{proof} We define an $\textup{Mp}(W)$-equivariant semilinear action of $G_{2'}$ on $(\omega_{\psi,X},C_c^\infty(Y,\mathcal{K}))$. For $\sigma \in G_{2'}$, there exists a unique element $\gamma \in (\mathcal{O}_F^\times)_{2'}$ such that $(\psi^{\gamma^2})^\sigma = \psi$. We set:
$$(r_\sigma \cdot f )(y)= \sigma\bigg( (\omega_{\psi,X}(m_\gamma,1) \cdot f) (y)\bigg).$$
Corollary \ref{cor:identification_omega_psi_and_omega_psi_gamma2} ensures $\omega_{\psi,X}(g,\lambda)$ and $r_\sigma$ commute for all $(g,\lambda) \in \textup{Mp}(W)$ and all $\sigma \in G_{2'}$. It is easy to see that the semilinear action thus defined is smooth. As $\textup{Spec}(\mathcal{K}) \to \textup{Spec}(\mathcal{L})$ is proétale, we are in the situation of an effective descent data of Theorem \ref{thm:pro-etale-descent} so the subspace $\mathcal{V}$ of $G_{2'}$-fixed points in $C_c^\infty(Y,\mathcal{K})$ satisfies $\mathcal{V} \otimes_\mathcal{L} \mathcal{K} \overset{\sim}{\to} C_c^\infty(Y,\mathcal{K})$. This $\mathcal{V}$ clearly inherits a smooth action of $\textup{Mp}(W)$. \end{proof}

\subsection{} We can improve this theorem by looking at a finer decomposition of $\omega_{\psi,X}$ involving the so-called even and odd parts of the Weil representation. We only pursue this goal when $R_0=\mathbb{Q}$ because the even and odd parts are known to be absolutely irreducible representation. Denote by
\begin{itemize}
\item  $(\omega_{\psi,X}^+,S_{\psi,X}^+)$ the even functions on $Y$;
\item $(\omega_{\psi,X}^-,S_{\psi,X}^-)$ the odd functions on $Y$.
\end{itemize}
Both are admissible absolutely irreducible representations of $\textup{Mp}(W)$ that can be realised over $\mathcal{L}$ by Theorem \ref{thm:weil-representation-descent-p-not-2}, where $\mathcal{L}/\mathbb{Q}$ has degree $|G_2| = \textup{val}_2(p-1)$. We let $p^*$ be $-p$ if $p \equiv 3 [4]$ and $p$ if $p \equiv 1 [4]$. We rephrase a well-known fact about the dependence of the isomorphism class of $\omega_{\psi,X}$ on the character $\psi$ in the following way:

\begin{theo} \label{thm:character-fields-of-odd-and-even-parts} $R(\omega_{\psi,X}^+) = R(\omega_{\psi,X}^-) = \left\{ \begin{array}{cl} \mathbb{Q} & \textup{ if } q \in p^{2 \mathbb{N}} \\ \mathbb{Q}[\sqrt{p^*}] & \textup{ if } q \in p^{2 \mathbb{N}+1} \end{array} \right.$. \end{theo}

\begin{proof} Via the identification of $G$ with a subgroup of $\mathcal{O}_F^\times$, let $\sigma \in G$ and let $\gamma \in \mathcal{O}_F^\times$ be the unique element such that $\psi^\sigma = \psi^\gamma$. The map $\sigma \mapsto \lambda$ induces a group isomorphism with $\mathbb{Z}/(p-1)\mathbb{Z}$ or $\mathbb{Z}_p^\times$, as embedded subgroups in $\mathcal{O}_F^\times$, according to whether $\textup{char}(F)$ is positive or not. We know by Section \ref{sec:weil_rep_dependance_on_psi} that $\omega_{\psi^\gamma,X}^\pm \simeq \omega_{\psi,X}^\pm$ if and only if $\gamma \in \mathcal{O}_F^{\times 2}$.

Therefore $H(\omega_{\psi,X}^\pm) = \{ \sigma \in G \ | \ \omega_{\psi^\sigma,X}^\pm \simeq \omega_{\psi,X}^\pm \} = G \cap \mathcal{O}_F^{\times 2}$. The subgroup $G \cap \mathcal{O}_F^{\times 2}$ of $G$ is equal to $G$ if and only if $q \in p^{2 \mathbb{N}}$, and is a subgroup of index $2$ otherwise. As a result, the fixed field of $G \cap \mathcal{O}_F^{\times 2}$ gives the character field, which is either $\mathbb{Q}$ or $\mathbb{Q}[\sqrt{p^*}]$ according to whether $q$ is a square or not. \end{proof}

\subsection{} We first deal with the even part, as it is simpler.

\begin{theo} \label{thm:descent-even-part-p-not-2} The even part $\omega_{\psi,X}^+$ can be realised over its character field $R(\omega_{\psi,X}^+)$. \end{theo}

\begin{proof} We first consider the case $q \equiv 3 [4]$ or equivalently $p \equiv 3 [4]$ and $q \in p^{2 \mathbb{N}+1}$. Note that we already have that $\mathcal{L} = \mathbb{Q}[\sqrt{-p}] = R(\omega_{\psi,X}^+)$ and therefore the theorem holds.

We can assume $-1 \in \mathcal{O}_F^{\times 2}$ as this is equivalent to $q \equiv 1 [4]$. Therefore $\Omega_{1,-1}^\psi = 1$ and $(-1,\gamma)_F=1$ for all $\gamma \in F^\times$. In particular $\omega_{\psi,X}^+((-\textup{id}_W,1)) = \textup{id}_{S_{\psi,X}^+}$ and
$$\omega_{\psi,X}^+(m_{- \gamma},1) = \omega_{\psi,X}^+(m_\gamma,1) \textup{ for all } \gamma \in F^\times.$$ 
we recall that $G \cap \mathcal{O}_F^{\times 2} \subseteq G$ is a subgroup of index $1$ or $2$ according to whether $q$ is a square or not. Let $\sigma \in G \cap \mathcal{O}_F^{\times 2} \mapsto \gamma \in \mathcal{O}_F^\times$ be any map subject to the relation $\sigma = \gamma^2$. 

Because $p$ is odd, the restriction of the quadratic Hilbert symbol to $\mathcal{O}_F^\times \times \mathcal{O}_F^\times$ is trivial. Since $\omega_{\psi,X}^+(m_{-\gamma},1) = \omega_{\psi,X}^+(m_\gamma,1)$, the earlier descent formula
$$(r_\sigma \cdot f )(y)= \sigma\bigg( (\omega_{\psi,X}^+(m_\gamma,1) \cdot f) (y)\bigg)$$
still defines a continuous semi-linear action of $G \cap \mathcal{O}_F^{\times 2}$ on $\omega_{\psi,X}^+$. So the representation can be realised over the fixed field of $G \cap \mathcal{O}_F^{\times 2}$ which is $R(\omega_{\psi,X}^+)$. \end{proof}

\subsection{} The odd part requires more work.

\begin{theo} \label{thm:descent-odd-part-p-not-2} The odd part $\omega_{\psi,X}^-$ can be realised over
\begin{itemize}
\item $R(\omega_{\psi,X}^-)$ if $p \equiv 3 [4]$ and $q \in p^{2\mathbb{N} + 1}$;
\item $R(\omega_{\psi,X}^-)[\sqrt{-p}]$ and $m(\omega_{\psi,X}^-) = 2$ otherwise.
\end{itemize} \end{theo}

\begin{proof} The easiest case to deal with is $q \equiv 3 [4]$ or equivalently $p \equiv 3 [4]$ and $q \in p^{2\mathbb{N}+1}$. In this case $\mathcal{L} = \mathbb{Q}[\sqrt{-p}]$ already, and Theorem \ref{thm:character-fields-of-odd-and-even-parts} ensures this is the character field.

To deal with the other two cases, we use information about the endomorphism ring of $\omega_{\psi,X}^-$. First of all, for all $\mathbb{Q} \subseteq \mathcal{M} \subseteq \mathcal{L}$, we have
$$\omega_{\psi,X}^-|_\mathcal{M} \otimes_\mathcal{M} \mathcal{L} \simeq \bigoplus_{\sigma \in \textup{Gal}(\mathcal{L}/\mathcal{M})} \omega_{\psi^\sigma,X}^-.$$
Recall that $G_2 = \textup{Gal}(\mathcal{L}/\mathbb{Q})$ is a finite cyclic group of order $2^k$ where $k=\textup{val}_2(p-1) \geq 2$. According to Theorem \ref{thm:character-fields-of-odd-and-even-parts}, two cases now arise
\begin{itemize}
\item $\omega_{\psi^\sigma,X}^- \simeq \omega_{\psi,X}^-$ for all $\sigma \in G_2$ \textit{i.e.} $G \cap \mathcal{O}_F^{\times 2} = G$ \textit{i.e.} $q$ is a square;
\item there exists $\sigma \in G_2$ such that $\omega_{\psi^\sigma,X}^-$ is not isomorphic to $\omega_{\psi,X}^-$.
\end{itemize}
In particular this distinction implies that the character field is $\mathbb{Q}$ or $\mathbb{Q}[\sqrt{p^*}]$.

We first prove a negative result: it is impossible to descend further than $\mathcal{L}$ along $\mathcal{L}$. 

\begin{lem} \label{lem:descent-impossible-along-mathcal-L} The representation $\omega_{\psi,X}^- \in \textup{Rep}_\mathcal{L}(\textup{Mp}(W))$ is not defined over any strict subextension of $\mathcal{L}$. \end{lem}

\begin{proof} Since $\mathcal{L}/\mathbb{Q}$ is cyclic, the extension $\mathcal{L}/\mathcal{L}'$ is cyclic for any subextension $\mathcal{L}'$ of $\mathcal{L}$. By cyclicity, if any descent can be achieved along $\mathcal{L}$, then it can be achieved over the totally real subextension $\mathcal{L}_0$ of index $2$ of the CM-field $\mathcal{L}$. Let $\tau \in \textup{Gal}(\mathcal{L}/\mathcal{L}_0) \subseteq \textup{Gal}(\mathcal{L}/\mathbb{Q})$ be the complex conjugation, which is the unique element of order $2$.

If $\omega_{\psi,X}^-$ descends to $\mathcal{L}_0$, there is a $\tau$-linear $\textup{Mp}(W)$-automorphism of $\omega_{\psi,X}^-$ of order $2$. Note that $\tau$-linear $\textup{Mp}(W)$-automorphisms are unique up to a scalar thanks to Schur's lemma. Choose $i \in \mathcal{O}_F^\times$ such that $i^2=-1$. Then
$$r_\tau \cdot f = \tau(\omega_{\psi,X}^-(m_i,1)f)$$
is a $\tau$-linear $\textup{Mp}(W)$-automorphism of $\omega_{\psi,X}^-$. It satisfies $r_\tau^2 = \omega_{\psi,X}^-(m_{-1},1) = - \textup{id}_{S_{\psi,X}^-}$.

Remark that any $\lambda r_\tau$ with $\lambda \in \mathcal{L}$ satisfies 
$$(\lambda r_\tau)^2 = - \lambda \tau(\lambda) \textup{id}_{S_{\psi,X}^-}.$$
But $\lambda \tau(\lambda) \in \mathbb{R}_+$ since $\mathcal{L}$ is a CM-field, so there is no $\tau$-linear $\textup{Mp}(W)$-automorphism of order $2$ and $\omega_{\psi,X}^-$ does not descend to any strict subextension of $\mathcal{L}$. \end{proof}

Let $\sigma$ be a generator of $G_2$. Then $G_2 \cap \mathcal{O}_F^{\times 2} = \langle \sigma^a \rangle$ where $a=1$ if $q$ is a square and $a=2$ otherwise. We set $\tau = \sigma^a$ and choose $\alpha \in \mathcal{O}_F^\times$ such that $\tau=\alpha^2$. The element $\tau$ has order $2^{k_a}$ where $k_a=k-a+1$ is $k$ or $k-1$. Define 
$$r_\tau \cdot f = \tau(\omega_{\psi,X}^-(m_\alpha,1)f).$$
Then $r_\tau^{2^{k_a}} = \omega_{\psi,X}^-(m_{-1},1) = - \textup{id}_{S_{\psi,X}^-}$.

To simplify notations we set $R$ for the character field, which is also the fixed field of $G_2 \cap \mathcal{O}_F^{\times 2}$ in $\mathcal{L}$. Let $A = \textup{End}_{R[\textup{Mp}(W)]}(\omega_{\psi,X}^-|_R)$. On the one hand $A$ contains $\mathcal{L}$ and $r_\tau$, so it contains the twisted $R$-algebra generated by $\mathcal{L}$ and $r_\tau$ which is of the form 
$$A_\tau = \mathcal{L}'[X_\tau]/(X_\tau^{2^{k_a}}+1).$$
On the other hand $\omega_{\psi,X}^-|_R \otimes_R \mathcal{L} \simeq 2^{k_a} \cdot \omega_{\psi,X}^-$, therefore
$$\textup{End}_{\mathcal{L}[\textup{Mp}(W)]}(\omega_{\psi,X}^-|_R \otimes_R \mathcal{L}) \simeq \mathcal{M}_{2^{k_a}}(\mathcal{L}).$$
Comparing the latter with the dimension of $A_\tau \otimes_R \mathcal{L}$, we deduce that $A_\tau \otimes_R \mathcal{L} \simeq \mathcal{M}_{2^{k_a}}(\mathcal{L})$. Therefore $A_\tau = A$ is a central simple $R$-algebra.

\begin{lem} Let $D/R$ be the unique quaternion division algebra ramified at $p$ and $\infty$. Then $A$ is isomorphic to $\mathcal{M}_{2^{k_a-1}}(D)$. \end{lem}

\begin{proof} We know that $A = A_\tau = \mathcal{L}'[X_\tau]/(X_\tau^{2^{k_a}}+1)$. We show that $A \otimes_\mathbb{Q} \mathbb{Q}_v$ is split if $v$ is a place different from $p$ or $\{\infty\}$. Note that $R_v = R \otimes_\mathbb{Q} \mathbb{Q}_v$ is an unramified extension of $\mathbb{Q}_v$ since $\mathcal{L}/\mathbb{Q}$ is only ramified at $p$, and likewise $\mathcal{L}_v = \mathcal{L} \otimes_\mathbb{Q} \mathbb{Q}_v$ is unramified. In particular, there exists $\lambda_v \in \mathcal{L}_v$ such that $N_{\mathcal{L}_v/R_v}(\lambda_v) = -1$ by local class field theory. As a result $\lambda_v r_\tau$ has order $2^{k_a}$ and we obtain a semi-linear action of $\langle \tau \rangle = H_2 \cap \mathcal{O}_F^{\times 2}$ on the Weil representation, which can be realised as $\mathcal{V} \in \textup{Irr}_R(\textup{Mp}(W))$ and therefore 
$$(\omega_{\psi,X}^- |_R) \otimes_R R_v \simeq (\omega_{\psi,X}^- \otimes_\mathcal{L} \mathcal{L}_v)|_{R_v}  \simeq 2^{k_a} \mathcal{V}.$$
This implies that $A \otimes_R R_v$ is split.

We claim that $A \simeq A^\textup{op}$. Indeed, since $-1$ is a square in $\mathcal{O}_F^\times$, we have $\omega_{\psi^{-1},X}^- \simeq \omega_{\psi,X}^-$. Furthermore, we have the well-known result 
$$(\omega_{\psi,X}^-)^\vee \simeq \omega_{\psi^{-1},X}^-.$$
When $V$ is a representation, there is a canonical identification $\textup{End}(V^\vee) \cong \textup{End}(V)^{\textup{op}}$. Therefore $A$ is isomorphic to $A^{\textup{op}}$ so the order of $A$ in the Brauer group is either $1$ or $2$.

Since $R$ is a number field, the $2$-torsion in the Brauer group of $R$ is generated by quaternion algebras. Therefore there exists a quaternion algebra $D$ over $R$ such that $A \simeq \mathcal{M}_{2^{k_a-1}}(D)$. We proved that $A$ splits at all places of $\mathbb{Q}$ different from $p$ and $\infty$. But $A$ can't be split by Lemma \ref{lem:descent-impossible-along-mathcal-L}, otherwise we would be able to descend the Weil representation to a strict subextension of $\mathcal{L}$. Because $R$ is totally real and $R/\mathbb{Q}$ is totally ramified at $p$, this implies that $D$ is the unique quaternion algebra ramified at $p$ and $\infty$. The lemma is now proved. \end{proof}

Since $\omega_{\psi,X}^-|_R$ is semisimple and its endomorphism ring $A$ is isomorphic to $\mathcal{M}_{2^{k_a-1}}(D)$, there exists $\mathcal{V} \in \textup{Irr}_R(\textup{Mp}(W))$ such that
$$\omega_{\psi,X}^-|_R \simeq 2^{k_a-1} \mathcal{V}.$$
Morever $\textup{End}_{R[\textup{Mp}(W)]}(\mathcal{V}) \simeq D$ and $\mathcal{V} \otimes_R \mathcal{L} \simeq 2 \omega_{\psi,X}^-$.

Note that $D$ splits over the quadratic extension $R[\sqrt{-p}]$ of $R$, therefore there exists $\mathcal{V}' \in \textup{Irr}_{R[\sqrt{-p}]}(\textup{Mp}(W))$ such that $\mathcal{V} \otimes_R R[\sqrt{-p}] \simeq 2 \mathcal{V}'$. Going to the composite $\mathcal{L}[\sqrt{-p}]$ of $R[\sqrt{-p}]$ and $\mathcal{L}$, we get that $\mathcal{V}'$ is a realisation of $\omega_{\psi,X}^- \otimes_\mathcal{L} \mathcal{L}[\sqrt{-p}]$ over $R[\sqrt{-p}]$ and the Schur index $m(\omega_{\psi,X}^-)$ must be $2$ as $\omega_{\psi,X}^-$ can't be realised over $R$. \end{proof}

\subsection{} In the modular setting $R_0 = \mathbb{F}_\ell$ with $\ell \neq 2$, we still have the decomposition into even and odd functions of the Weil representation, though these representations may fail to be irreducible. The obvious analogue of Theorem \ref{thm:character-fields-of-odd-and-even-parts} is valid \textit{i.e.}
$$R(\omega_{\psi,X}^\pm) = \left\{ \begin{array}{cl} \mathbb{F}_\ell & \textup{ if } q \in p^{2 \mathbb{N}} \\ \mathbb{F}_\ell[\sqrt{p^*}] & \textup{ if } q \in p^{2 \mathbb{N}+1} \end{array} \right.$$
where $\mathbb{F}_\ell(\sqrt{p^*}) = \mathbb{F}_\ell$ if $p^*$ is a square. Likewise, the proof of Theorem \ref{thm:descent-even-part-p-not-2} still works. However, we note the following major difference with the $R_0 = \mathbb{Q}$ case: as a consequence of Wedderburn's Theorem, the Schur index is always $1$ when $R_0 = \mathbb{F}_\ell$. In particular $\omega_{\psi,X}^-$ can be realised over its character field. Indeed, in Theorem \ref{thm:descent-odd-part-p-not-2}, the obstruction to descent comes from a norm problem which can always be solved for finite fields by surjectivity of the norm, as we can always find $\lambda \in \mathcal{L}$ such that 
$$N_{\mathcal{L}/R(\omega_{\psi,X}^-)}(\lambda) = -1.$$

\subsection{} We simply remark that most of the arguments we developed could be applied in families \textit{i.e.} for ring of integers of number fields. We point out that Theorem \ref{thm:weil-representation-descent-p-not-2} is still valid over $\mathcal{O}_\mathcal{L}[1/p]$ as $\mathcal{O}_\mathcal{L}[1/p] \to \mathcal{O}_\mathcal{K}[1/p]$ is proétale and $\omega_{\psi,X}$ can be realised over $\mathcal{O}_\mathcal{K}[1/p]$. Therefore the Weil representation can be realised over $\mathcal{O}_\mathcal{L}[1/p]$. If we invert $2$ as well, we can use the decomposition $\omega_{\psi,X} \cong \omega_{\psi,X}^+ \oplus \omega_{\psi,X}^-$ and our descent arguments still work over the localised version $\mathcal{O}[1/2p]$ of the rings of integers of the fields appearing in Theorems \ref{thm:descent-even-part-p-not-2} and \ref{thm:descent-odd-part-p-not-2}.

\section{Descent when $p = 2$}

\subsection{} Let $F$ be a $2$-adic field. As opposed to the previous $p \neq 2$ case, the restriction of $( \ , \ )_F$ to $\mathcal{O}_F^\times \times \mathcal{O}_F^\times$ is no longer trivial. For example $(-1,-1)_{\mathbb{Q}_2}=-1$. The Galois group $G = \textup{Gal}(\mathcal{K}/R_0)$ is an open subgroup of $\mathbb{Z}_2^\times \subseteq \mathcal{O}_F^\times$. We have $\mathbb{Z}_2^{\times 2} = 1 + 8 \mathbb{Z}_2$ so the cardinality of $\mathbb{Z}_2^\times / \mathbb{Z}_2^{\times 2}$ is $4$. However, unlike the case $p \neq 2$, the pro-order of $\mathbb{Z}_2^\times$ only has one single prime divisor and the torsion elements in $\mathbb{Z}_2^\times$ are simply $1$ and $-1$. Another key difference is the fact that $\mathbb{Z}_2^\times$ is not procyclic as $\mathbb{Z}_2^\times /\mathbb{Z}_2^{\times 2} \simeq \mathbb{Z}/2\mathbb{Z} \times \mathbb{Z}/2\mathbb{Z}$. However $\mathbb{Z}_2^{\times 2}$ is procyclic. There are three maximal procyclic subgroups, namely $1+(2+4)\delta + 8\mathbb{Z}_2$ and $1+ 2 \delta + 8\mathbb{Z}_2$ and $1+4 \delta+ 8 \mathbb{Z}_2$ where $\delta$ runs over $\{0,1\}$.

\subsection{} We define a descent argument on $G_{2'} = \mathbb{Z}_2^{\times 2} \cap \textup{Gal}(\mathcal{K}/R_0)$. Note that $G_{2'}$ has pro-order $2^\infty$, but we use this notation by analogy with the previous case. Let $\mathcal{L} = \mathcal{K}^{G_{2'}}$. The degree of $\mathcal{L}/R_0$ is $1$, $2$ or $4$. When $R_0 = \mathbb{Q}$, the extension $\mathcal{L}/\mathbb{Q}$ has degree $4$ and is biquadratic as $\mathcal{L} = \mathbb{Q}[\sqrt{-1},\sqrt{2}] = \mathbb{Q}[\zeta_8]$.

\begin{theo} \label{thm:weil-representation-descent-p-is-2} The Weil representation can be realised over $\mathcal{L}$. \end{theo}

\begin{proof} The situation is rather similar to the case $p \neq 2$, but there are a few technical complications. In order to define our descent argument, we first want to be able to extract roots from $G_{2'}$ to $\mathcal{O}_F^\times$ \textit{i.e.} to define a group morphism $\sigma \in G_{2'} \mapsto \lambda \in \mathcal{O}_F^\times$ subject to the relation $\sigma = \lambda^2$. To do so, we can embed $G_{2'}$ in $1+ 2 \delta + 8\mathbb{Z}_2$ -- note that $1+4 \delta+ 8 \mathbb{Z}_2$ works equally well -- and extract roots in this subgroup of $\mathbb{Z}_2^\times \subseteq \mathcal{O}_F^\times$ \textit{i.e.} for all $\sigma \in G_{2'}$ there exists a unique $\lambda \in 1+ 2 \delta + 8\mathbb{Z}_2$ such that $\sigma = \lambda^2$. We obtain a group isomorphism $\sigma \in G_{2'} \mapsto \lambda \in \mathcal{O}_F^\times$ and we denote by $G'$ its image in $\mathcal{O}_F^\times$, which contains $G_{2'}$ as an index $2$ subgroup.

As opposed to the case $p \neq 2$, the quadratic Hilbert symbol for $\lambda, \lambda' \in \mathcal{O}_F^\times$ in the action of $M(X)$ is non-trivial \textit{i.e.} $\omega_{\psi,X}(m_\lambda,1) \omega_{\psi,X}(m_{\lambda'},1) = (\lambda,\lambda')_F \omega_{\psi,X}(m_{\lambda \lambda'},1)$. We can consider the restriction of $(-,-)_F$ to $G'$, which is a subgroup of $\mathcal{O}_F^\times$. We define an action of $G_{2'}$ on the Weil representation via
$$r_\sigma \cdot f  = \sigma(\omega_{\psi,X}(m_\lambda,\gamma_\lambda) f)$$
where $\partial \gamma$ trivialises the $2$-cocycle $(- , - )_F$ on $G'$ \textit{e.g.} we can take
\begin{itemize}
\item $\gamma_u = 1$ for $u \in G'$ when $(-,-)_F$ is trivial on $G'$;
\item $\gamma_u = 1$ for $u \in G_{2'}$ and $\gamma_u = i$ for $u \in G' \backslash G_{2'}$ otherwise.
\end{itemize}
This defines an effective descent data as in the proof of Theorem \ref{thm:weil-representation-descent-p-not-2} and we can realise the Weil representation over the fixed field of $G_{2'}$ \textit{i.e.} over $\mathcal{L}$. \end{proof}

\subsection{} We now focus on the case $R_0=\mathbb{Q}$. Let $\{1,-1,3,5\}$ be representatives of $\mathbb{Z}_2^\times / \mathbb{Z}_2^{\times 2}$ in $\mathbb{Z}_2^\times$ and write $[\alpha] = \langle \alpha, \mathbb{Z}_2^{\times 2} \rangle$ for $\alpha \in \mathbb{Z}_2^\times$. Let $\mathbb{Z}_2^{\times} \supseteq A = \mathcal{O}_F^{\times 2} \cap \mathbb{Z}_2^\times \supseteq \mathbb{Z}_2^{\times 2}$. We simply recall that $\zeta_8 = \frac{1+i}{\sqrt{2}}$ and $\mathbb{Q}[\zeta_8] = \mathbb{Q}[\sqrt{-2},\sqrt{-1}]$ contains three quadratic extensions of $\mathbb{Q}$.

\begin{theo} \label{thm:descent-even-part-p-is-2} The even part $\omega_{\psi,X}^+$ can be realised over its character field, which is \begin{itemize}
\item $\mathbb{Q}$ if $A=\mathbb{Z}_2^\times$;
\item $\mathbb{Q}[\sqrt{-2}]$ if $A=[3]$;
\item $\mathbb{Q}[\sqrt{-1}]$ if $A=[5]$;
\item $\mathbb{Q}[\sqrt{2}]$ if $A=[-1]$;
\item $\mathbb{Q}[\zeta_8]$ if $A=\mathbb{Z}_2^{\times 2}$.
\end{itemize} \end{theo}

\begin{proof} Let $A=\mathbb{Z}_2^\times$ and choose $\sqrt{-1}$ and $\sqrt{3}$ in $\mathcal{O}_F^\times$. For $\sigma_{-1}$ and $\sigma_3$ the Galois elements corresponding to $-1$ and $3$, we define 
$$r_{\sigma_{-1}} \cdot f = \sigma_{-1} (\omega_{\psi,X}^+(m_{\sqrt{-1}},1)f) \textup{ and } r_{\sigma_3} \cdot f = \sigma_3 (\omega_{\psi,X}^+(m_{\sqrt{3}},1)f).$$
Since the actions of $\sigma_{-1}$ and $\sigma_3$ commute, we obtain a semilinear Galois action of $\textup{Gal}(\mathbb{Q}[\zeta_8]/\mathbb{Q})$, so $\omega_{\psi,X}^+$ descends to $\mathbb{Q}$.

Let $A=[3]$ and choose $\sqrt{3}$ in $\mathcal{O}_F^\times$. Define 
$$r_{\sigma_3} \cdot f = \sigma_3 (\omega_{\psi,X}^+(m_{\sqrt{3}},\gamma)f)$$
where $\gamma \in \mathbb{Q}[\zeta_8]$ satisfies $\sigma_3(\gamma) \gamma = (\sqrt{3},\sqrt{3})_F = (-1,\sqrt{3})_F$. We can take $\gamma = 1$ when $(-1,\sqrt{3})_F=1$ and $\gamma=\zeta_8$ when $(-1,\sqrt{3})_F=-1$. We obtain a semilinear Galois action of $\textup{Gal}(\mathbb{Q}[\zeta_8]/\mathbb{Q}[\sqrt{-2}])$ as $\zeta_8^3+\zeta_8 = \sqrt{-2}$, so $\omega_{\psi,X}^+$ descends to $\mathbb{Q}[\sqrt{-2}]$, which is also the character field because $\omega_{\psi^{-1},X}^+$ is not isomorphic to $\omega_{\psi,X}^+$.

The case $A=[5]$ is similar to $A=[3]$. We omit the details of the proof.

Let $A=[-1]$ and choose $\sqrt{-1}$ in $\mathcal{O}_F^\times$. Define
$$r_{\sigma_{-1}} \cdot f = \sigma_{-1} (\omega_{\psi,X}^+(m_{\sqrt{-1}},1)f).$$
We obtain a semilinear Galois action of $\textup{Gal}(\mathbb{Q}[\zeta_8]/\mathbb{Q}[\sqrt{2}])$ as $\zeta_8^{-1}+\zeta_8 = \sqrt{2}$, so $\omega_{\psi,X}^+$ descends to $\mathbb{Q}[\sqrt{2}]$, which is also the character field as $\omega_{\psi^3,X}^+$ is not isomorphic to $\omega_{\psi,X}^+$.

Let $A=\mathbb{Z}_2^{\times 2}$. Then $\omega_{\psi^\alpha,X}^+$ for $\alpha \in \{-1,3,5\}$ is not isomorphic to $\omega_{\psi,X}^+$, so the character field is $\mathbb{Q}[\zeta_8]$, which was already a field of realisation for $\omega_{\psi,X}^+$. \end{proof}

\subsection{} Once again, the odd part requires a bit more work.

\begin{theo} \label{thm:descent-odd-part-p-is-2} The representation $\omega_{\psi,X}^-$ can be realised over  \begin{itemize}
\item $\mathbb{Q}[\sqrt{-2}]$ or $\mathbb{Q}[\sqrt{-1}]$ if $A=\mathbb{Z}_2^\times$ and its character field is $\mathbb{Q}$;
\item its character field $\mathbb{Q}[\sqrt{-2}]$ if $A=[3]$;
\item its character field $\mathbb{Q}[\sqrt{-1}]$ if $A=[5]$;
\item $\mathbb{Q}[\zeta_8]$ if $A=[-1]$ and its character field is $\mathbb{Q}[\sqrt{2}]$;
\item its character field $\mathbb{Q}[\zeta_8]$ if $A=\mathbb{Z}_2^{\times 2}$.
\end{itemize}
The Schur index is $2$ if $A = \mathbb{Z}_2^\times, [-1]$ and $1$ otherwise. \end{theo}

\begin{proof} Let $A=\mathbb{Z}_2^\times$. The character field is $\mathbb{Q}$. However, the representation $\omega_{\psi,X}^-|_K$ where $K$ is either $\mathbb{Q}[\sqrt{-1}]$ or $\mathbb{Q}[\sqrt{-2}]$ satisfies $\omega_{\psi,X}^-|_K \otimes_K \mathbb{Q}[\zeta_8] \simeq 2 \omega_{\psi,X}^-$. Moreover $\textup{End}_{K[\textup{Mp}(W)]}(\omega_{\psi,X}^-|_K) \simeq K'[\langle \sigma \rangle]/(\sigma^2 - 1)$ where $\langle \sigma \rangle = \textup{Gal}(\mathbb{Q}[\zeta_8]/K)$ and define
$$r_\sigma \cdot f = \sigma(\omega_{\psi,X}^-(m_{\sqrt{\sigma}},1)f)$$
where $(\sqrt{\sigma},\sqrt{\sigma})_F = (-1,\sqrt{\sigma})_F = 1$. Note that $\textup{End}_{K[\textup{Mp}(W)]}(\omega_{\psi,X}^-|_K) \simeq \mathcal{M}_2(K)$ and therefore $\omega_{\psi,X}^-|_K$ has length $2$ and is semisimple \textit{i.e.} $\omega_{\psi,X}^-$ descends to $K$.

Let $A=[3]$ and choose $\sqrt{3}$ in $\mathcal{O}_F^\times$. Define 
$$r_{\sigma_3} \cdot f = \sigma_3 (\omega_{\psi,X}^-(m_{\sqrt{3}},\gamma)f)$$
where $\gamma \in \mathbb{Q}[\zeta_8]$ satisfies $\sigma_3(\gamma) \gamma = (\sqrt{3},\sqrt{3})_F = (-1,\sqrt{3})_F$. We can take $\gamma = 1$ when $(-1,\sqrt{3})_F=1$ and $\gamma=\zeta_8$ when $(-1,\sqrt{3})_F=-1$. We obtain a semilinear Galois action of $\textup{Gal}(\mathbb{Q}[\zeta_8]/\mathbb{Q}[\sqrt{-2}])$ as $\zeta_8^3+\zeta_8 = \sqrt{-2}$, so $\omega_{\psi,X}^-$ descends to $\mathbb{Q}[\sqrt{-2}]$, which is also the character field because $\omega_{\psi^{-1},X}^-$ is not isomorphic to $\omega_{\psi,X}^-$.

The case $A=[5]$ is similar to $A=[3]$. We omit the details of the proof.

Let $A=[-1]$ and choose $\sqrt{-1}$ in $\mathcal{O}_F^\times$. The character field is $\mathbb{Q}[\zeta^8]^A = \mathbb{Q}[\sqrt{2}]$. Define
$$r_{\sigma_{-1}} \cdot f = \sigma_{-1} (\omega_{\psi,X}^-(m_{\sqrt{-1}},1)f).$$
Then $\textup{End}_{\mathbb{Q}[\sqrt{2}][\textup{Mp}(W)]}(\omega_{\psi,X}^-|_{\mathbb{Q}[\sqrt{2}]}) \simeq \mathbb{Q}[\zeta_8]'[\langle \tau \rangle] / (\tau^2+1)$ where $\langle \tau \rangle = \textup{Gal}(\mathbb{Q}[\zeta_8]/\mathbb{Q}[\sqrt{2}])$. This is a central simple algebra over $\mathbb{Q}[\sqrt{2}]$ which is ramified at $\sqrt{2}$ and $\infty$, indeed it is the scalar extension of $\mathbb{Q}[\sqrt{-1}]'[\textup{Gal}(\mathbb{Q}[\sqrt{-1}]/\mathbb{Q})]/(\tau^2+1) = D_{-1,-1}$ to $\mathbb{Q}[\sqrt{2}]$. As a result, the Schur index of $\omega_{\psi,X}$ is $2$ and a field of realisation is $\mathbb{Q}[\zeta_8]$.

Let $A=\mathbb{Z}_2^{\times 2}$. Then $\omega_{\psi^\alpha,X}^-$ for $\alpha \in \{-1,3,5\}$ is never isomorphic to $\omega_{\psi,X}^-$, so the character field is $\mathbb{Q}[\zeta_8]$, which is already a field of realisation for $\omega_{\psi,X}^-$.

When a representation can be realised over its character field, its Schur index is $1$. The Schur index is $2$ in the remaining cases because the ring of endomorphisms of the restriction to the character field is a quaternion division algebra. \end{proof}

\subsection{} In the modular setting $R_0 = \mathbb{F}_\ell$, note that $\ell \neq p = 2$ as $\ell \neq p$, so we still have the decomposition into even and odd functions of the Weil representation, though these representations may fail to be irreducible. The obvious analogue of Theorem \ref{thm:descent-odd-part-p-is-2} is valid -- \textit{i.e.} replacing $\mathbb{Q}$ by $\mathbb{F}_\ell$ -- allowing $\mathbb{F}_\ell[\alpha] = \mathbb{F}_\ell$ if $\alpha$ already belonged to $\mathbb{F}_\ell$. As noted in the $p \neq 2$ case, Wedderburn's Theorem ensures the character field is always a field of realisation. Similarly, the analogue of Theorem \ref{thm:descent-odd-part-p-is-2} is valid, replacing $\mathbb{Q}$ by $\mathbb{F}_\ell$.

\subsection{} Once again, we remark that most of the arguments we developed could be applied in families \textit{i.e.} for ring of integers of number fields. We point out that Theorem \ref{thm:weil-representation-descent-p-is-2} is still valid over $\mathcal{O}_\mathcal{L}[1/p]$ as $\mathcal{O}_\mathcal{L}[1/2] \to \mathcal{O}_\mathcal{K}[1/2]$ is proétale and $\omega_{\psi,X}$ can be realised over $\mathcal{O}_\mathcal{K}[1/2]$. Therefore the Weil representation can be realised over $\mathcal{O}_\mathcal{L}[1/p]$. Moreover $2$ is invertible, so we can use the decomposition $\omega_{\psi,X} \cong \omega_{\psi,X}^+ \oplus \omega_{\psi,X}^-$ and our descent arguments still work over the localised version $\mathcal{O}[1/p]$ of the rings of integers of the fields appearing in Theorems \ref{thm:descent-even-part-p-is-2} and \ref{thm:descent-odd-part-p-is-2}.

\bibliographystyle{alpha}
\bibliography{lesrefer}

\end{document}